\documentclass[10pt]{article}
\usepackage[paperwidth=8.1in,paperheight=11in,top=1.1in, bottom=1.11in, left=1.00in, right=1.00in]{geometry}
\usepackage{amsthm,amsmath,amssymb,mathtools}  
\usepackage{chngcntr}
\usepackage{amsmath}  
\usepackage{graphicx}                                                         
\usepackage{caption}
\usepackage{subfigure}
\usepackage{enumitem}                                                       
\usepackage{hyperref}
\usepackage[mathscr]{eucal}                                            
\usepackage[usenames,dvipsnames,svgnames,table]{xcolor}
\usepackage[normalem]{ulem}                                         
\usepackage{pstricks}
\usepackage{rotating}
\usepackage{lscape}
\usepackage{dsfont}
\usepackage{dsfont}
\usepackage{tikz}
\mathtoolsset{showonlyrefs=true}                                    
\theoremstyle{plain}
\usepackage{relsize}
\newtheorem{theorem}{Theorem}
\numberwithin{theorem}{section}

\newtheorem{lemma}[theorem]{Lemma}          
 \newtheorem{proposition}[theorem]{Proposition}

\theoremstyle{definition}

\newtheorem{notation}[theorem]{Notation}

\newtheorem{remark}[theorem]{Remark}

\def \y {{\eta}}
\def \s {{\sigma}}
\def \g {{\gamma}}
\def \a {{\alpha}}
\def \b {{\beta}}
\def \m {{\mu}}

\def \l {{\lambda}}
\def \f {{\varphi}}

\def \G {{\Gamma}}

\def \t {{\tau}}

\def \d {{\delta}}
\newcommand\eps{\varepsilon}
\def \p {{\partial}}

\newcommand\R{\mathbb{R}}
\newcommand\Eb{E}%{\mathbb{E}}
\newcommand\Lc{\mathscr{L}}
\newcommand\Nb{\mathbb{N}}
\def \F {\mathcal{F}}
\newcommand\Pb{\mathbb{P}}

\newcommand{\<}{\langle}
\renewcommand{\>}{\rangle}
\renewcommand{\(}{\left(}
\renewcommand{\)}{\right)}
\renewcommand{\[}{\left[}
\renewcommand{\]}{\right]}

\def \uuu {{h}}%{{\bar{{\bar{u}}}_\s}}

\definecolor{blue}{rgb}{0,0,1}

\definecolor{violet}{rgb}{1,0,1}

\begin{document}

\title{Strong regularization by noise for a class of kinetic SDEs driven by symmetric $\alpha$-stable processes}

\author{
Giacomo Lucertini
\thanks{Dipartimento di Matematica, Universit\`a di Bologna, Piazza di Porta S. Donato 5, 40126 Bologna, Italy.
Email: {\tt giacomo.lucertini3@unibo.it}.}
\and
St\'ephane Menozzi
\thanks{Laboratoire de Modélisation Mathématiques d'Evry (LAMME), UMR CNRS 8071, Universit\'e d'Evry Val d'Essonne-Paris Saclay, 23  Boulevard de France, 91037 Evry, France.
Email: {\tt  stephane.menozzi@univ-evry.fr}.}
\and
Stefano Pagliarani
\thanks{Dipartimento di Matematica, Universit\`a di Bologna, Piazza di Porta S. Donato 5, 40126 Bologna, Italy.
Email: {\tt stefano.pagliarani9@unibo.it}.}
%This work was partially supported by
}
\date{This version: \today}
\maketitle

\begin{abstract}

We establish strong well-posedness for a class of degenerate SDEs of kinetic type with autonomous diffusion driven by a symmetric $\alpha$-stable process under H\"older regularity conditions for the drift term. 
We partially recover the thresholds for the H\"older regularity that are optimal for weak uniqueness. 
In general dimension, we only consider $\alpha=2$ and need an additional integrability assumption 
for the gradient of the drift: this condition is satisfied by Peano-type functions. 
In the one-dimensional case we do not need any additional assumption. 
In the multi-dimensional case, the proof is based on a first-order Zvonkin transform/PDE, 
while for the one-dimensional case we use a second-order Zvonkin/PDE transform together with a Watanabe-Yamada technique. 
\newline

\newline

{\noindent\textbf{Acknowledgements}:  The research of G.L. and S.P was partially supported by the INdAM - GNAMPA project CUP\_J53D23003800006. The research of S.P was partially supported by the PRIN22 project CUP\_E53C23001670001. The research of S.M. and S.P was partially supported by the INdAM - GNAMPA project CUP\_E53C23001670001.}
\end{abstract}
%%%%%% 60H10: SDEs
%%%%%% 60H50: regularization by noise

%%%%%%%%%%%%%%%%%%%
%		INTRODUCTION		%
%%%%%%%%%%%%%%%%%%%
%
\section{Introduction}

We are concerned with proving strong well-posedness for degenerate systems of stochastic differential equations of the following type:
\begin{align}\label{eq:sde}
%\begin{cases}
  dV_{t}&=\m(t,V_{t})dt+\s(t,V_{t-})dL^{(\alpha)}_{t},\\ \label{eq:sde_x}
  dX_{t}&=(V_{t} + F(t,X_t) ) dt,	
%\end{cases}
\end{align}
where 
\begin{itemize}
\item $(L^{(\alpha)}_t)_{t\in[0,T]}$ is a $d$-dimensional symmetric (i.e. rotationally invariant) $\alpha$-stable L\'evy process with $\alpha \in (1,2]$ {(namely, we address the sub-critical regime)};
\item $\sigma:[0,T] \times \R^d \to \mathcal{M}^{d\times d}$ and $\mu,F:[0,T] \times \R^d \to \R^d$ are measurable functions.
\end{itemize}

In particular, the pair $(V_t,X_t)$ takes values on $\R^d\times\R^d=\R^{2d}$, although the noise $L^{(\alpha)}$ directly acts on the first $d$ components. Note also that, for $\alpha=2$, $L^{(\alpha)}$ is a $d$-dimensional Brownian motion, while, for $\alpha\in(1,2)$, $L^{(\alpha)}$ is a pure-jump L\'evy process generated by the fractional Laplacian
\begin{equation}\label{eq:gen_lap_frac}
\Delta^{\frac{\alpha}{2}} g (v) := \frac{1}{2} \int_{\R^d}  \big[ g(v+w) + g(v-w) - 2 g(v)  \big]  |w|^{-d-\alpha} dw .
\end{equation}

%with characteristic $\infty$-triplet $(0,0,\nu)$, where $\nu(dw) = |w|^{-d - \alpha} dw$.
%and \eqref{eq:sde}-\eqref{eq:sde_x}-\eqref{eq:sde_x}-\eqref{eq:sde_x} is an It\^o stochastic differential equation. 
Hereafter, we refer to Eq. \eqref{eq:sde}-\eqref{eq:sde_x} as \emph{kinetic-type} SDE because it can be viewed as a perturbation of the kinetic system obtained by setting $F\equiv 0$ in \eqref{eq:sde_x}. In this case, $(V_t,X_t)$ can be interpreted as the pair velocity-position of a particle subject to stochastic force. 

The reason for studying kinetic-type systems like \eqref{eq:sde}-\eqref{eq:sde_x} is twofold. On the one hand, models of this form appear in the literature in several applications ({physics/hydrogeology: \cite{MR3671993}, mathematical finance: \cite{MR2791231}}, {among others: \cite{delarue2010density} and references therein}). On the other hand, \eqref{eq:sde}-\eqref{eq:sde_x} is a prototypical example of a more general class of systems, where the noise directly acts on a set of components and propagates to the others by means of H\"ormander-type conditions on the drift.

Note that the jump-diffusion stochastic differential equation %in the first component of 
$\eqref{eq:sde}$ is autonomous, i.e. its coefficients do not depend on $X$, thus the uniqueness of $V$ is fairly well understood. In particular, we can rely on a {vast} literature on \emph{regularization-by-noise} phenomena, namely when the drift coefficient $\mu$ is irregular, say H\"older continuous or even distributional, but pathwise uniqueness for the equation can be restored thanks to the action of the noise.

%The second component in
Equation \eqref{eq:sde_x}, however, is actually an ordinary differential equation with stochastic source. Clearly, by the classical Cauchy theory, if the component $V$ is unique then %$F\equiv 0$, then %$X$ is simply the integral of $V$, and thus its 
the uniqueness of the component $X$ is trivial so long as $F(t,\cdot)$ is Lipschitz continuous. On the other hand, when $F(t,\cdot)$ is rougher, regularization-by-noise pehonomena are not obvious as the impact of the noise in the dynamics of $X$ is weaker than in the dynamics of $V$. Indeed, it is known that some \emph{regularity thresholds} on $F$ arise in order to establish uniqueness results for $X$.  For instance, in \cite{MR3808994} and {\cite{MR4591369}} it was shown that weak uniqueness is ensured if $F(t,\cdot)$ is $\beta$-H\"older continuous with $\beta>1/{(1+\alpha)}$, while it may fail if $\beta<1/{(1+\alpha)}$. On the other hand, Zvonkin transform-based techniques, relying on Schauder estimates for the semigroup,  allow to prove pathwise uniqueness, in general, only if $\beta>(1+\alpha/2)/(1+\alpha)$ (\cite{MR4498506}, \cite{MR4124429}). In this paper we close the gap between these two thresholds,  namely we prove that pathwise uniqueness holds so long as $\beta>1/(1+\alpha)$, for two particular cases of \eqref{eq:sde}-\eqref{eq:sde_x}: (i) ``always" in dimension one ($d=1$); (ii) when $F(t,\cdot)$ is a ``Peano-like" function and $\alpha=2$ (Brownian noise case), regardless of the dimension ($d\in\mathbb{N}$).

\subsection{Well-posedeness of kinetic-type SDEs: state of the art}

Consider the system:
\begin{equation}\label{eq:sde_non}
\begin{cases}
  dV_{t}=\m(t,V_{t},X_t)dt+\s(t,V_{t-},X_t)dL^{(\alpha)}_{t},\\
  dX_{t}=(V_{t} + F(t,X_t) ) dt,	
\end{cases}
\end{equation}
which is slightly more general than \eqref{eq:sde}-\eqref{eq:sde_x} in that the first equation is no longer autonomous. 

%We consider coefficients that are rough, namely H\"older continuous in the $v$ or $x$ variable. 
%We are able to prove pathwise uniqueness for \eqref{eq:sde} under two different setting of assumptions, depending on the dimension $d$. 
%
%For a general $d\geq1$, we consider a class of drift coefficients $F$ that are differentiable almost everywhere, with derivative that satisfies 
%an integrability estimate. %Although this hypothesis is not easily understandable, 
%We already remark that Peano type functions are covered by this class: for example,
%\begin{equation}
%F(x)=|x|^\b,\qquad x\in\R^d,
%\end{equation}
%where $\b>1/3$.
%
%When $d=1$ we can consider a larger class of drift functions $F$, namely $\b$-H\"older function, with $\b>1/3$. 
%This result partially recover the weak uniqueness thresholds obtained in \cite{MR4498506}, \cite{MR4358660} and \cite{MR3808994}. 
%
%In \eqref{eq:sde} the non-degenerate component (or velocity in kinetic models), namely $V_t$, is an autonomous process. 
%The regularity assumed for $\m$ and $\s$ is the one needed for establishing the strong uniqueness for the $V_t$ component. 
%As we will explain in \red{Remark/Section \ref{}}, difficulties arises trying to consider coefficients $\m$ and $\s$ that are also dependent on the variable $x$: 
%%difficulties arises when we try to consider also an $x$ dependence in the $\m$ and $\s$: 
%this difficulty is technical and it is related to the strategy of proof we developed.

{
Many results can be extended replacing the equation for $X_t$ with the more general one
\begin{equation}
\label{MORE_GEN_DEG}
dX_{t}=G(t,V_{t},X_t)  dt,
\end{equation}
under a weak H\"ormander-type condition on $G$ (see for example \cite{MR3606742} or \cite{MR4498506}).
}
%\begin{remark}
%Our proof techniques should also hold replacing the equation for $X_t$ with the more general one
%\begin{equation}
%\label{MORE_GEN_DEG}
%dX_{t}=G(t,V_{t},X_t)  dt,
%\end{equation}
%under weak H\"ormander-type condition on $G$ (see for example \cite{MR3606742} or \cite{MR4498506}).
%\end{remark}

%\subsection{Literature on strong uniqueness of (non-)degenerate SDEs}
The current work can be seen as an illustration of the regularization by noise phenomenon. The idea behind is that adding a noise to an \textit{a priori} ill-posed differential equation can somehow restore some form of uniqueness. We will here focus on {pathwise} uniqueness for the corresponding SDE, which thanks to the {Yamada-Watanabe theorem} \cite{MR0288876} actually amounts, once weak existence is established, to prove {strong well-posedness}.
To derive such a property, Zvonkin introduced a seminal idea in his work \cite{MR0336813} consisting in expressing the time integral of a \textit{bad drift} in terms of a corresponding \textit{smoother} PDE (which actually benefits from some kind of parabolic bootstrap). Through this approach, known as the \textit{Zvonkin transform}, he obtained therein strong well-posedness for scalar Brownian diffusions with bounded drifts. These results were extended to the multidimensional setting by Veretennikov \cite{MR0568986}.

Another step was achieved by Krylov and Röckner \cite{MR2117951}, who managed to obtain strong well-posedness for Brownian SDEs with additive {noise} and inhomogeneous drifts in Lebesgue spaces satisfying a Serrin type condition. This result was later generalized by Zhang \cite{MR3010120} who considered similar drifts with multiplicative diffusion coefficient in an appropriate Sobolev space. Namely, the generalized first order derivative is assumed to have appropriate integrability conditions (enlarging the \textit{usual} Lipschitz setting for the diffusion coefficient in the multi-dimensional case). Similar results were obtained in the stable driven case by Xie and Zhang \cite{MR4058986}. In this setting, the mere Serrin type condition is not enough to restore strong well-posedness and some additional regularity, in appropriate Bessel potential spaces, is needed for the drift. Let us quote as well various recent works concerning the critical Serrin type condition \cite{MR4040996}, \cite{MR4348687}, \cite{rock:zhao:21}.

Let us also mention, in the {non-Markovian setting}, a rough-path type approach to strong uniqueness, see e.g. Catellier and Gubinelli \cite{MR3505229}, who derived for SDEs driven by an additive fractional Brownian motion, strong well-posedness for a drift with smoothness index strictly greater than $1-1/(2H)$ where $H$ stands for the Hurst parameter of the driving noise, which is also the corresponding self-similarity index.

For a more complete \textit{tour} about regularization by noise we can refer to the survey/monograph by Flandoli  \cite{MR2796837}, see also \cite{MR2593276} for connections with some transport PDEs.

We have up to now mentioned results concerning how adding a noise could restore well-posedness for a {non-degenerate} system. Let us also indicate that the regularization by noise phenomenon also occurs when handling kinetic type (or Hamiltonian) dynamics as the first order term in \eqref{eq:sde_non}. Of course, in that setting, since the noise acts through its integral, it is somehow damped and the noise can restore well-posedness for drifts that are less rough than in the non-degenerate setting. In the current kinetic case of equation \eqref{eq:sde_non}, and even for the more general dynamics \eqref{MORE_GEN_DEG} for the degenerate component, it was established in the Brownian setting, by Chaudru de Raynal \cite{MR3606742},  through the Zvonkin type approach that strong uniqueness holds provided   the drift in the degenerate variable is H\"older with regularity index greater than $2/3$ (see also \cite{MR3511355} for an extension to Dini continuous coefficients or \cite{MR3661662}, \cite{MR3833741} for a kinetic model with a drift in Bessel potential space with \textit{sharp} index 2/3). Observe that this threshold actually matches the one appearing in \cite{MR3505229} replacing $H$ by the current self similarity index of the integral of the Brownian motion, namely $1-1/(2 \times \frac 32) =\frac 23$. On the other hand, in the  same setting, it was established in \cite{MR3808994} %by the same author 
that weak well-posedness holds provided the spatial regularity in the non-degenerate variable is greater than $1/3$ in the whole drift coefficient. This result was then relaxed in \cite{MR4358660} where it was proved that this smoothness condition is actually needed only for the coefficient of the degenerate component. Let us specify that %the indicated paper 
{\cite{MR4358660}} actually addresses weak well-posedness for a somehow general model of noise propagating through a chain of differential equations whose coefficients satisfy a non-degeneracy assumption allowing the noise to transmit from one level to another.
For the strong uniqueness related to such models, the corresponding problem was addressed in \cite{MR4498506}. We can also refer to \cite{MR4588621} for similar results in the non-Markovian fractional setting. 

Note that, in Brownian kinetic setting, the threshold $1/3$ for weak well-posedness seems to actually be \textit{almost} sharp in the sense that some counter-examples to weak well-posedness were provided in \cite{MR3808994}. Similar issues have been discussed as well  in the pure jump stable case. Let us mention \cite{MR4591369} for weak uniqueness of degenerate systems (including the model considered in the paper). The threshold appearing is then $1/(1+\alpha) $, i.e. there is continuity as well in the stability index for weak well-posedness, and the threshold is also \textit{almost sharp}, {in that some counterexamples to weak well-posedness can be built if the drift coefficient is strictly less than $1/(1+\alpha)$ H\"older continuous.}
Eventually, the strong well-posedness in the pure-jump kinetic stable driven case was discussed in \cite{MR4124429}, where it is proved that strong uniqueness holds, in the H\"older setting,  as soon as the  drift   has regularity greater than  $(1+\alpha/2)/(1+\alpha) $ in the degenerate variable for $\alpha>1 $. There is again some kind of continuity with the thresholds appearing in the Brownian case.

\subsection{%Setting and main result
Main results and open problems}

Hereafter, throughout the paper, $T>0$ is a fixed time-horizon. 
First we state a structural assumption for the system \eqref{eq:sde}-\eqref{eq:sde_x}, which is a non-degeneracy condition on the diffusion of the $V$ component.

\begin{itemize}
\item[$\text{[{\bf Coe}]}$] (coercivity of $\sigma\s^*$) There exists $\Lambda>0$ such that
	\begin{equation}
	\Lambda^{-1}|\y|^2\leq\<\s\s^*(t,v)\y,\y\>\leq\Lambda|\y|^2, \qquad %\textcolor{blue}{\forall}
	 v,\y\in\R^{d},
	\label{eq:holdsigma}
	\end{equation}
	for almost every $t\in[0,T]$.
\end{itemize}

To state the regularity assumptions on the coefficients we introduce the following H\"older spaces:
%\begin{assumption}[coercivity of $\sigma\s^*$]\label{as:coer}
%There exists $\Lambda>0$ such that
%	\begin{equation}
%	\Lambda^{-1}|\y|^2\leq\<\s\s^*(t,v)\y,\y\>\leq\Lambda|\y|^2, \qquad (v,\y)\in\R^{2d},
%	\label{eq:holdsigma}
%	\end{equation}
%	for almost every $t\in[0,T]$.
%\end{assumption}
%We recall the following
\begin{notation}
Let $\beta\in (0,1]$. 
\begin{itemize}
\item $C^\b%=C^\a(\R^d)
$ and  $C^\b_b$
denote the spaces of %$\a$-H\"older continuous functions on $\R^d$:
functions $g:\R^d \to \R$ such that
	 \begin{equation}
	 |g|_{C^\b}:=\sup_{x\neq y}\frac{|g(x)-g(y)|}{|x-y|^\b}<+\infty  , \qquad \|g\|_{C^\b_b}:=|g|_{C^\b}+\sup_{x\in\R^d}|g(x)|<+\infty  ,
	 \end{equation}
	 respectively, {i.e. $C^\beta $ denotes the \textit{homogeneous} H\"older space while $C_b^\beta$ includes the supremum norm}.
	%When $\a=1$ then $C^\a$ denotes the space of Lipschitz continuous functions
%\item $C^\a_b%=C^\a(\R^d)\cap L^\infty(\R^d)
%$ denotes the space of bounded functions of $C^\a$:
%	 \begin{equation}
%	 \|g\|_{C^\a_b}=|g|_{C^\a_b}+\sup_{x\in\R^d}g(x)|<+\infty \ .
%	 \end{equation}
\item $L^\infty_TC^\b$ and $L^\infty_TC^\b_b$ denote the spaces of measurable functions $f=f(t,\cdot)\colon[0,T]\to C_b(\R)$ such that
\begin{equation}
	\sup\limits_{t\in[0,T]}|f(t,\cdot)|_{C^\b}+\sup\limits_{t\in[0,T]}|f(t,0)| < +\infty, \qquad %\| f \|_{L^\infty_TC^\b_b}:= 
	\sup\limits_{t\in[0,T]}\|f(t,\cdot)\|_{C^\b_b} < +\infty,
	\end{equation}
	respectively.
%\item $L^\infty_TC^\a:= L^\infty(]0,T[;C^\a)$ denotes the space of real functions $f\colon]0,T[\times\R^d\to\R$ such that 
%	\begin{equation}
%	\sup_{t\in]0,T[}|f(t,\cdot)|_{C^\a} < +\infty.
%	\end{equation}
%	$L^\infty_TC^\a_b:=L^\infty(]0,T[;C^\a_b)$ denotes the space of real functions $f\colon]0,T[\times\R^d\to\R$ such that 
%	\begin{equation}
%	\sup_{t\in]0,T[}\|f(t,\cdot)\|_{C^\a} < +\infty.
%	\end{equation}
%	$f(\cdot,x)$ is bounded uniformly in $x\in\R^d$ and
%	$f(t,\cdot)$ is $\a$-H\"older continuous functions on $\R^d$ uniformly in $t\in]0,T[$.
%\item $$ denotes the space of functions $f(t,\cdot)$ that are bounded continuous on $\R^d$ uniformly in $t\in]0,T[$	
\end{itemize}
\end{notation}
%The following are regularity hypotheses on the drift coefficients $\mu$ and $F$.

%\begin{itemize}
%\item[$\text{[{\bf H-$\sigma$-$\gamma$}]}$] There exists $\gamma\in ]0,1]$ such that $\s\in L^\infty_TC^\g_b$.
%\item[$\text{[{\bf H-$\mu$-$\gamma$}]}$] There exists $\gamma\in ]0,1[$ such that $\m\in L^\infty_TC^\g_b$.
%\item[$\text{[{\bf H-$F$-$\gamma$}]}$] There exists $\gamma\in ]0,1[$ such that $F\in L^\infty_TC^{\frac{1+\g}{\alpha+1}}$ and $F(\cdot,0)\in L^\infty([0,T])$.
%\end{itemize}

\begin{remark}
Let $\alpha\in (1,2]$. Under Hypothesis [{\bf Coe}] and assuming %[{\bf H-$\mu$-$\gamma$}]-[{\bf H-$F$-$\gamma$}]-[{\bf H-$\sigma$-$\gamma$}], 
$\mu,\sigma \in L^\infty_TC^\g_b$, $F \in L^\infty_TC^{\frac{1+\g}{1+\alpha}}$ for some $\gamma\in(0,1)$, \eqref{eq:sde}-\eqref{eq:sde_x} is weakly well-posed (see {\cite{MR4591369}} in the pure-jump case {and \cite{MR4358660} for $\alpha=2 $}%\footnote{\textcolor{blue}{Cito l'articolo più generale perché in \cite{MR3808994} Chaudru usa Zvonkin e le ipotesi sono le stesse sulla regolarità pert il drift degenere e quello non degenere.}}
).
\end{remark}
In this paper %$\mu \in L^\infty_TC^\g_b$ and 
$F \in L^\infty_TC^{\frac{1+\g}{1+\alpha}}$, for an arbitrary $\gamma>0$, is the basic regularity assumption on the coefficient in \eqref{eq:sde_x}. {Concerning the coefficients $\mu$ and $\sigma$ in \eqref{eq:sde}, different regularity assumptions will be required depending on the dimension $d$ and on the value of $\alpha$ (see Theorem \ref{th:well_posedness_autonomous})}.
%Together with the additional assumption $\sigma\in L^\infty_TC^\g_b$ and with the structural hypothesis [{\bf H}], it yields weak well-posedness of \eqref{eq:sde} (\cite{}). 
%The reason why we did not include the H\"older regularity assumption on $\sigma$ in [{\bf H-$\gamma$}] is that stronger hypotheses will be required in order to prove strong well-posedness (see [{\bf H-$1$}]-[{\bf H-$2$}] below). 
%We now present two different setting of assumptions, respectively for the multi-dimensional case and the one-dimensional case. 

We now introduce a last condition on the drift coefficient $F$, which will be required in order to prove pathwise uniqueness in the multi-dimensional diffusive case. First we introduce the following
\begin{notation}\label{not:moll}
Let $f:[0,T]\times \R^d \to \R$, for any $\eps>0$ we denote by $f^{(\eps)}:[0,T]\times \R^d \to \R$ the function
\begin{equation}\label{eq:moll}
f^{(\eps)}(t,x) : = \int_{\R^d} \Phi_{\eps}(x-y) f(t,y) dy, \qquad \Phi_{\eps}(y) = \frac{1}{\eps^d} {\Phi\left (\frac y\eps\right)},
\end{equation}
where $\Phi$ is a (fixed) smooth probability density on $\R^d$ with compact support.
\end{notation}
%two alternative sets of assumptions. The validity of either one of them, combined with [{\bf H}] and [{\bf H-$\gamma$}], will yield pathwise uniqueness for \eqref{eq:sde}.

\begin{itemize}
\item[$\text{[{\bf Pea}]}$] %The process $L^{(\alpha)}$ is a Brownian motion, i.e. $\a=2$, and $\s\in L^\infty_TC^1_b$. Furthermore, 
%there exists $\gamma\in ]0,1[$ such that
%\begin{itemize}
%\item[(i)] $\s\in L^\infty_TC^1_b$ and $\m\in L^\infty_TC^\g_b$;
%\item[(ii)] $F\in L^\infty_TC^{\frac{1+\g}{3}}$ and $F(\cdot,0)\in L^\infty([0,T])$;
%\item[(iii)] 
%For almost every $s\in[0,T]$, the function $x\mapsto F(s,x)$ is differentiable almost everywhere and %it satisfies the following estimate:
For any $\lambda>0$, there exists a non-negative $g\in L^1([0,T])$ such that
	\begin{align}\label{eq:moment_estimate}
	\int_{\R^{d}}|\nabla_y F^{(\eps)}(s,y)|\Gamma(\lambda
	(s-t)^3,y-\theta)dy \leq g(s-t)%C_T (s-t)^{%\frac{3}{2}(\frac{1+\g}{3}-1)-\beta}
	,\\%\qquad 
	0\leq t < s \leq T,\quad \theta\in\R^d, \quad \eps>0.
	\end{align}
Here {$\Gamma=\Gamma(t,x)$} is the $d$-dimensional Gaussian density
	\begin{equation}\label{eq:standard_gauss}
	{\Gamma(t,x)=\frac{1}{(2\pi t)^{\frac{d}{2}}} \ e^{-\frac{1}{2}\frac{|x|^2}{t}}, \qquad (t,x)\in(0,\infty)\times\R^{d}.}
	\end{equation}
%\end{itemize}
\end{itemize}

We refer to the latter as to a \emph{Peano-like} assumption, in that the simplest non-trivial example of H\"older continuous function that satisfies [{\bf Pea}] is given by (see Proposition \ref{prop:peano}) 
\begin{equation}\label{eq:ex_peano}
F(t,x):= a(t) |x|^{\beta}, \qquad \beta \in(1/3,1], \qquad a\in L^{\infty}([0,T]),
\end{equation}
giving rise to a Peano-like differential equation for $X$. We highlight that $\beta>1/3$ coincides with the threshold for weak uniqueness previously established in \cite{MR3808994}.
The probabilistic intuition of Hypothesis [{\bf Pea}] is clarified in the following
\begin{remark}\label{rem:bound_mean_nab_Xeps}
Let $\alpha=2$, namely $L^{(\alpha)}$ is a Brownian motion. Then, under Hypothesis [{\bf Coe}] and assuming %[{\bf H-$\mu$-$\gamma$}]-[{\bf H-$F$-$\gamma$}]-[{\bf H-$\sigma$-$\gamma$}], 
$\mu,\sigma \in L^\infty_TC^\g_b$, $F \in L^\infty_TC^{\frac{1+\g}{3}}$ for some $\gamma\in(0,1)$, condition [{\bf Pea}] ensures that
\begin{equation}\label{eq:bound_mean_nab_Xeps}
\int_t^T \Eb[| \nabla_x F^{(\eps)}(s, X^{\eps,t,v,x}_s ) | ] ds \leq C_T, \qquad (t,v,x)\in [0,T]\times\R^{2d}, \quad \eps>0,
\end{equation}
with $X^{\eps,t,v,x}_s$ denoting the second component of the solution to the regularized version of \eqref{eq:sde}-\eqref{eq:sde_x}, starting at time $t$ from $(v,x)$. This is a direct consequence of the density estimates in \cite[Theorem 1.1 and Remark 1.3]{MR4554678} (see also {\eqref{eq:density_upper_bound_1dim} in} Proposition \ref{prop:density_bound} below), namely the density of $X^{\eps,t,v,x}_s$ is bounded by the Gaussian density appearing in \eqref{eq:moment_estimate}.
%
%\begin{equation}%\label{eq:}
%p_X(t,z;T,y) \leq C_T \Gamma\big( (s-t)^3,y- \theta_{t,T}(z) \big) dy  
%\end{equation}
\end{remark}

We are now in position to present the main result.
\begin{theorem}\label{th:main}
Let Hypothesis \emph{[{\bf Coe}]} be in force. Let also $\alpha\in (1,2]$ and %\emph{[{\bf H-$\mu$-$(1-{\alpha}/{2}+\theta)$}]-[{\bf H-$F$-$\gamma$}]} be in force
assume
\begin{equation}%\label{eq:}
\mu, \sigma \in L^\infty_TC^\g_b, \qquad F \in L^\infty_TC^{\frac{1+\g}{1+\alpha}},
\end{equation}
%$\mu \in L^\infty_TC^\g_b$, $F \in L^\infty_TC^\g$ 
for some %$\theta\in(0,\alpha/2]$ and 
$\gamma\in(0,1)$. %and assume that Eq. \eqref{eq:sde} is strongly well-posed, in the sense that strong existence and pathwise uniqueness hold. 
Finally, let one of the following additional assumptions hold: 
\begin{itemize}
\item[(i)] the functions $\sigma, \mu$ and $F$ are scalar-valued, i.e. $d=1$;
%, and one of the following conditions holds:
%\begin{itemize}
%\item[(i-a)] $\sigma \in L^\infty_TC^{1/\alpha}_b$;%\emph{[{\bf H-$\sigma$-$1/\alpha$}]} holds true.
%\end{itemize}
\item[(ii)] Hypothesis  \emph{[{\bf Pea}]} is in force, and the process $L^{(\alpha)}$ is a {$d$-dimensional} Brownian motion ($\a=2$), {$d\ge 1 $}.
%, and $\sigma \in L^\infty_TC^{1}_b$.
\end{itemize}
If Eq. \eqref{eq:sde} is strongly well-posed, then %the system \eqref{eq:sde}-
so is Eq. \eqref{eq:sde_x}, in the sense that strong existence and pathwise uniqueness hold.
%\emph{[{\bf H.1}]} or \emph{[{\bf H.2}]} are also satisfied, then pathwise uniqueness holds for the system \eqref{eq:sde}.
\end{theorem}

\begin{remark}
In Theorem \ref{th:main}, strong well-posedness of the autonomous SDE \eqref{eq:sde} is taken by assumption. Precise state-of-the-art conditions on $\mu$ and $\sigma$ for this, also in relation to $d$ and $\alpha$, are reported in Theorem \ref{th:well_posedness_autonomous} below. We stress that the proof of Theorem \ref{th:main} (well-posedness of Eq. \eqref{eq:sde_x}) does not rely on these additional hypotheses.
\end{remark}

{\begin{remark}\label{rem:mu_reg}
In Theorem \ref{th:main}-(ii), the regularity assumption on the drift coefficient $\mu$ could be relaxed, %so as to assume only measurability and boundedness, 
for the basic estimates in \cite{MR4554678}, which we rely on in the proof, hold when $\mu$ is only measurable and with sub-linear growth. However, the focus in this work is mainly on the regularity of the coefficient $F$, rather than of $\mu$. Therefore, we prefer the current statement to promote ease of reading.
\end{remark}}

Theorem \ref{th:main} closes the gap between the two existing H\"older thresholds, namely $\gamma>0$ and $\gamma> \alpha / 2$ for weak and pathwise uniqueness of \eqref{eq:sde}-\eqref{eq:sde_x}, respectively, in two different settings.

\vspace{2pt}

\emph{Setting (i)}. In dimension one, the gap is completely closed, when the driving noise is a symmetric $\alpha$-stable Levy process with $\alpha \in (1,2]$. The case of Brownian noise is thus included. The technique we employ to prove pathwise uniqueness comprises two main steps: first, we apply a Zvonkin-type deterministic transform ({also known as} It\^o-Tanaka trick) to a given solution $(V_t,X_t)$, in order to re-write the dynamics of $X_t$ without the \emph{irregular} drift function $F$. Second, we employ a technique introduced by Ikeda-Watanabe (\cite{ikeda2014stochastic}), which allows to reduce the regularity requirements on the (new) noise coefficient in order to conclude with a Gronwall-type argument. This step is performed by designing a suitable approximation of the absolute value function {and is therefore very dimension dependent}. 

In practice, with the first step one reduces the pathwise uniqueness problem to establishing suitable regularity estimates for the solution to the Zvonkin-type PDE
\begin{equation}\label{eq:cauchy_zvonkin}
\begin{cases}
\big[ \partial_t + \Lc_v + (v + F(t,x))\partial_x - \lambda \big] u(t,v,x) = -F(t,x) , \quad (t,v,x)\in(0,T)\times \R^{2},\\%\times \R^d,\\
 u_{\lambda} (T,\cdot,\cdot) \equiv 0,
\end{cases}
\end{equation} 
for $\lambda>0$, where $\Lc_v$ is the generator of the (autonomous) component $V$. 

This step alone is not sufficient to prove pathwise uniqueness for $\gamma\in(0,1)$. Indeed, the standard way to proceed prescribes to re-parametrize the dynamics of two solutions $X,X'$ in terms of $u_{\lambda}$, and then apply Burkholder-Davis-Gundy-type inequalities %{} 
to the process $|X_t-X_t'|^2$. Following \cite{MR4124429}, one finds that the regularity estimates on $u_{\lambda}$ needed to successfully carry out a Gronwall-type argument require $\gamma > \alpha/2$.

With the second step, we reduce the regularity requirements on $u_{\lambda}$. %in order to prove pathwise uniqueness. 
When $d=1$, alternative to applying BGD-type estimtes, we approximate $|X_t - X'_t|$ with a suitable smooth approximation $\varphi_n (X_t - X'_t)$, which allows us to %and simply observe that the stochastic integral that appears in the transformed dynamics is a true martingale, and thus has null expectation. This allows to 
run the Gronwall argument so long as $x\mapsto \nabla_v u_{\lambda} (t,v,x)$ is $1/2$-H\"older continuous. This is true in light of the estimate (see Proposition \ref{prop:Schauder} below)
\begin{equation}\label{eq:holder_gradient}
[ x\mapsto \nabla_v u_{\lambda} (t,v,x) ] \in C^{\frac{\alpha+\gamma}{\alpha+1}}_b \quad \text{uniformly w.r.t. } (t,v)\in (0,T)\times \R,
\end{equation}
recalling that $\alpha\in (1,2]$. 

%highest order estimate needed to successfully carry out a Gronwall-type argument is that the gradient of $u_{\lambda}$ with respect to the $v$-variable, $\nabla_v u_{\lambda} (t,v,x)$, is Lipschitz continuous with respect to $x$, uniformly in $(t,v)$. However, the optimal regularity is  (see Proposition \ref{th:zhang} below)
%\begin{equation}\label{eq:holder_gradient}
%[ x\mapsto \nabla_v u_{\lambda} (t,v,x) ] \in C^{\frac{\alpha+\gamma}{\alpha+1}}_b \quad \text{uniformly w.r.t. } (t,v)\in (0,T)\times \R^d,
%\end{equation}
%% for $x\mapsto u_{\lambda} (t,v,x)$ is $(\gamma+2)/({\alpha+1})$-H\"older continuous, 
%which implies Lipschitz continuity only if $\gamma \geq 1$.

\emph{The riddle of the exceeding regularity:} interestingly, by \eqref{eq:holder_gradient} there seems to be ``too much" regularity on $\nabla_v u_{\lambda}$, as 
\begin{equation}%\label{eq:}
\frac{\alpha+\gamma}{\alpha+1} > \frac{\alpha}{\alpha+1}>1/2, \qquad \gamma\in(0,1),\ \alpha\in(1,2].
\end{equation}
In order not to ``waste" regularity, one would be tempted to reduce the requirement $F\in L^\infty_TC^{\frac{1+\g}{\alpha+1}}$ by taking $\gamma<0$. However, this is not possible because \eqref{eq:holder_gradient} only holds for $\gamma\in(0,1)$. Note that the validity of \eqref{eq:holder_gradient} for $\gamma<0$, which would lead to pathwise uniqueness for \eqref{eq:sde_x}, would in turn contradict the counter-example in \cite{MR3808994} to weak-uniqueness. Therefore, the reason for the exceeding regularity in \eqref{eq:holder_gradient} must be sought in probabilistic terms. One possibility is that we are not taking fully advantage of the Ikeda-Watanabe technique, which could be adapted, in the kinetic case, to cover higher-dimensional cases. For instance, in \eqref{eq:sde_non} the $V$ component is no longer autonomous, and thus $(V,X)$ solves a $2$-dimensional fully coupled system (here $d=1$). This case cannot be handled with a smooth approximation of the $2$-dimensional Euclidean norm, but the approximation of different (non-Euclidean) norms could be attempted in order to take advantage of the maximal regularity of $u_{\lambda}$. We aim to investigate this direction in future research.

\vspace{2pt}

\emph{Setting (ii)}. For a general dimension $d$, in the purely diffusive case, the gap is closed so long as the drift coefficient $F$ satisfies an additional Peano-like assumption. Indeed, the regularity required on $F(t,v,x)$ along the $x$-variable is $(1+\gamma)/3$-Holder continuous, with $\gamma\in(0,1)$, which is the same needed for weak well-posedness, but Hypothesis [{\bf Pea}] is also required. The reason for the latter can be motivated by the {proof} technique we utilize. 

The idea is again getting rid of the irregular drift $F$ by means of a change of variables in terms of a function that possesses suitable regularity properties. Namely, we consider the solution to the following (regularized) transport equation with random coefficients
\begin{equation}\label{eq:transport_random}
\begin{cases}
\partial_t u_{\eps} (t,x) +  \nabla_x u_{\eps} (t,x) (F^{(\eps)}( t, x ) + V_t)    = F^{(\eps)}(t, x)     ,  \qquad  t<T  , \\
u_{\eps}(T,x) = 0.
\end{cases}
\end{equation}
Applying ($1$st order) It\^o formula then yields
\begin{align} 
X_T - X'_T =& \int_0^T  \nabla_x u_{\eps} (t,X_t)  ( F^{(\eps)}( t, X_t )-F(t, X_t ) ) dt \\&+ \int_0^T  \nabla_x u_{\eps} (t,X'_t) ( F^{(\eps)}( t, X'_t )-F(t, X'_t ) ) dt.
\label{eq:u_key_prop_eps_bis}
\end{align}
%{This approach can somehow be seen as a first order Zvonkin-type transform}.
Pathwise uniqueness can then be obtained, passing to the limit as $\eps\to 0^+$, by establishing a uniform $L^p$-bound on $\nabla_x u_{\eps} (t,X_t)$. To this end, we first represent $\nabla_x u_{\eps}$ in terms of the characteristic curves of the random vector-field in right-hand side of \eqref{eq:transport_random}, then apply a probabilistic argument that strongly relies on the bound \eqref{eq:bound_mean_nab_Xeps}, which can be established under [{\bf Pea}] (see Remark \ref{rem:bound_mean_nab_Xeps} above).

Note that the function $u_{\eps}$ can be understood as a \emph{first-order Zvonkin transform}. Instead of considering the second-order PDE \eqref{eq:cauchy_zvonkin}, associated to the stochastic system \eqref{eq:sde}-\eqref{eq:sde_x}, we consider the first-order PDE in \eqref{eq:transport_random}, associated to the single equation for $X$, understood as an ordinary differential equation with random coefficients.  As a result, there is no diffusion term in the dynamics of the transformed process $u_{\eps}(t, X_t)$, and pathwise uniqueness can be inferred only with a suitable bound on the gradient $\nabla_x u_{\eps}$, with no further regularity on the latter. We point out that the transformed dynamics are derived here by means of a pathwise It\^o formula. Probabilistic arguments intervene only to establish an $L^p$-bound on $\nabla_x u_{\eps} (t,X_t)$, as we strongly rely on uniform Gaussian upper bounds for the transition density of $X$.

This technique can be, in principle, extended to more general driving noises. For instance, in the case $\alpha <2$, the PDE for $u_{\eps}$ and the representation \eqref{eq:u_key_prop_eps_bis} would remain basically unchanged. However, to the best of our knowledege, we are not aware of uniform upper bounds on the density of $X$, which are needed to derive the $L^p$-bound on $\nabla_x u_{\eps} (t,X_t)$, in full generality but only under some dimensional constraints (see Huang and Menozzi \cite{MR3573301}).

\emph{Pathwise uniqueness for ``{truly}" H\"older functions:} Our result in this setting closes the gap on the threshold regularity between weak and pathwise uniqueness, only for a specific class of drift functions $F$. Namely, we are able to check Hypothesis [{\bf Pea}] for a family of functions whose gradient is continuous except for a countable number of singularities that are isolated enough, namely functions that behave locally like \eqref{eq:ex_peano}. If the singularities accumulate around a certain point the argument may break down (see Section \ref{sec:peano}). Therefore, this technique seems not to be suited to deal with H\"older functions $F$ whose gradient is not locally bounded, say distributional. The question of whether pathwise uniqueness for kinetic-type systems \eqref{eq:sde}-\eqref{eq:sde_x} holds or not, for a generic $F(t, \cdot) \in C^{\frac{1+\g}{1+\alpha}}$, remains open, although our result shows that possible counterexamples are not to be sought, at least in the diffusive framework (i.e. $\alpha=2$), among Peano-like models. 

\vspace{2pt} 

{The rest of the paper is organized as follows. Section \ref{sec:prel} contains some preliminary results, mainly needed for the proof of Theorem \ref{th:main}. In particular, in Section \ref{sec:autonomous} we recall some state-of-the art results for the strong well-posedness of the autonomous SDE \eqref{eq:sde} in the H\"older setting. In Section \ref{sec:density_estimates}, we recall a Gaussian upper bound (from \cite{MR4554678}) on the transition density of the (diffusive) system \eqref{eq:sde}-\eqref{eq:sde_x}, which is needed to prove Theorem \ref{th:main}-(ii). In Section \ref{sec:schauder} we derive, expanding on the Schauder estimates in \cite{MR4124429}, sharp regularity estimates for the gradient of the solution to \eqref{eq:cauchy_zvonkin}-like Kolmogorov equations, needed for the proof of Theorem \ref{th:main}-(i). Section \ref{sec:proof_th_one} is devoted to proving Theorem \ref{th:main}-(i). In particular, the key PDE transformations and their relevant regularity estimates are laid out in Section \ref{sec:regul_estimates}. Section \ref{sec:proof_th_one_core} contains the probabilistic core of the proof, and finally Section \ref{sec:proof_th_one_prop} contains the proof of a technical proposition. Section \ref{sec:multi_dim} focuses on Theorem \ref{th:main}-(ii). In particular, in Section \ref{sec:multi_dim_prelim} we integrate the statement with an analysis of the hypothesis [{\bf Pea}], with examples and counterexamples. Section \ref{sec:multi_dim_core} contains the core of the proof. The proof of a technical estimate, namely a uniform $L^p$-bound on $\nabla_x u_{\eps} (t,X_t)$ in \eqref{eq:transport_random}, is postponed until Section \ref{sec:proof_lemmas}.
} 
%Finally, the problem of pathwise uniqueness for \eqref{eq:sde} in the critical ($\alpha=1$) and sub-critical ($\alpha\in (0,1)$) cases are completely open. 

%\subsection{Structure of the paper}
%%%%%%%%%%%%%%%%%%%
%		MULTI DIMENSION		%
%%%%%%%%%%%%%%%%%%%
%

\section{Preliminaries}\label{sec:prel}

\subsection{SDEs driven by $\alpha$-stable noise%: the $1$-dimensional case
}
\label{sec:autonomous}

In this section we consider the autonomous equation %in the system 
\eqref{eq:sde}, which we re-write here for the reader's convenience:
\begin{equation}\label{eq:autonomous_SDE_stable}
dV_{t}=\m(t,V_{t})dt+\s(t,V_{t-})dL^{(\alpha)}_{t},
\end{equation}
where $(L^{(\alpha)}_t)_{t\in[0,T]}$ is a $d$-dimensional symmetric (i.e. rotationally invariant) $\alpha$-stable L\'evy process with $\alpha \in (1,2]$, and $\sigma:[0,T] \times \R^d \to \mathcal{M}^{d\times d}$, $\mu,F:[0,T] \times \R^d \to \R^d$ are measurable functions.
%\end{itemize}
%
%Here $\sigma:[0,T] \times \R \to \R$ and $\mu,F:[0,T] \times \R \to \R$ are measurable functions and $L^{(\alpha)}$ is a scalar symmetric (i.e. rotationally invariant) $\alpha$-stable L\'evy process with $\alpha\in(1,2)$, namely a Markov process generated by the fractional Laplacian $\Delta_v^{\alpha/2}$ in \eqref{eq:1d_frac_lap}.
%\begin{equation}%\label{eq:}
%\Delta^{\frac{\alpha}{2}} u(v) = \frac{1}{2} \int_{\R} \big[ u(v+w) - u(v-w) - 2u(v)  \big] |w|^{-1-\alpha} dw.
%\end{equation}

 We have the following result for the strong well-posedness of \eqref{eq:autonomous_SDE_stable}.
\begin{theorem}\label{th:well_posedness_autonomous}
%Let \emph{[{\bf H-$\sigma$-$1/\alpha$}]} and \emph{[{\bf H-$\mu$-$\gamma$}]} be in force for some $\gamma\in (0,1]$. 
 Let {Hypothesis \emph{[{\bf Coe}]} be in force,} $\alpha\in (1,2]$ and one of the following sets of assumptions be satisfied: 
\begin{itemize}
\item[(i)]  the functions $\sigma, \mu$ are scalar-valued, i.e. $d=1$, plus one of the following:
\begin{itemize}
\item[(a)] $\sigma \in L^\infty_TC^{1}_b$ and $\mu \in L^\infty_TC^{1-\frac{\alpha}{2}+\eta}_b$ for some $\eta>0$;
\item[(b)] the noise is additive ($\sigma(t,v)\equiv \sigma$) and $\mu \in L^\infty_TC^{\gamma}_b$ for some $\gamma\in (0,1)$; 
\item[(c)] $\sigma \in L^\infty_TC^{\frac{1}{\alpha}}_b$ and the drift is constant ($\mu(t,v) \equiv \mu$);
\item[(d)] process $L^{(\alpha)}$ is a Brownian motion ($\a=2$), $\sigma \in L^\infty_TC^{\frac{1}{2}}_b$ and $\mu \in L^\infty_TC^{\gamma}_b$ for some $\gamma\in (0,1)$;
\end{itemize}
\item[(ii)] process $L^{(\alpha)}$ is a Brownian motion ($\a=2$), $\sigma \in L^\infty_TC^{1}_b$ and $\mu \in L^\infty_TC^{\gamma}_b$ for some $\gamma\in (0,1)$.
\end{itemize}
Then the SDE \eqref{eq:autonomous_SDE_stable} is strongly well-posed, in the sense that strong existence and pathwise uniqueness hold.
\end{theorem}
For a proof to Theorem \ref{th:well_posedness_autonomous} we refer to: 
\cite[Theorem 1.1]{MR4328678} in the case (i)-(a), 
\cite[Theorem 3.1]{MR368146} in the case (i)-(b), 
\cite[Theorem 1.1]{MR1971592} in the case (i)-(c), 
\cite[Theorem 4]{MR0336813} in the case (i)-(d) %, \cite[Theorem ]{} in the case (ii). 
and (ii).
%\oran{Da GL a S.P.: nei casi (i)-(b) e (ii) mi sembra che il drift basti limitato, ma non è rilevante per i nostri scopi direi...}
\begin{remark}[One-dimensional case] \label{rem:one_d}
To the best of our knowledge, the assumptions (i)-(a-d) in Theorem \ref{th:well_posedness_autonomous} are the {weakest} available in the literature in the one-dimensional H\"older setting. Note that, in the diffusive case, {H\"older continuity is the only regularity requirement on the drift coefficient $\mu$}, % the only regularity threshold on the drift coefficient $\mu${\red { }assuming that it is at least H\"older continuous}, 
and the diffusion coefficient $\sigma$ is only $1/2$-H\"older continuous. On the other hand, in the $\alpha$-stable setting, stronger assumptions are needed on $\mu$ and $\sigma$. 

{We also refer to \cite{bass2001stochastic} and \cite{athreya2020strong} for strong well-posedness results in the one-dimensional setting for distributional drift.} {In particular, the assumption $\mu \in L^\infty_TC^{\gamma}_b$ in (i)-(d) could be weakened by requesting that $b$ is a distribution in a suitable space. However, the Schauder estimates of Proposition \ref{prop:Schauder}, which are necessary to prove our main result (Theorem \ref{th:main}-(i)), are not yet available for a distributional $\mu$. Therefore, although future extensions in this direction seem possible, in this work we prefer to consider only a H\"older setting.}
\end{remark}
%{\red
%\begin{remark}[On the regularity of $\mu$] \label{rem:mu_reg}
%Notice that, outside of the H\"older setting, the regularity assumptions for $\mu$ are not optimal and they could be weakened.  
%However, in sight of the main result, Theorem \ref{th:main}, we decide to restrict ourself to this framework in way that the statements could be uniform in notation. 
%In the Brownian cases, bounded measurability is a sufficient condition to obtain strong well-posedness of the SDE \eqref{eq:autonomous_SDE_stable}, 
%together with Proposition \ref{prop:density_bound} and the Schauder estimates in Proposition \ref{prop:Schauder} which will appear below. 
%But on the contrary, in the pure jump case these results, that are crucial to the purpose of our proof, are not available in literature. 
%We believe that Theorem \ref{th:main} can be extended beyond the H\"older setting for $\mu$, 
%but we leave it to a future work since the main focus of this paper is indeed the regularity of $F$. 
%\end{remark}
%}

\subsection{Density estimates for diffusive kinetic-type systems}\label{sec:density_estimates}

We recall an upper bound on the density of the system \eqref{eq:sde}-\eqref{eq:sde_x} with $\alpha=2$, namely when $L^{(\alpha)}$ is a $d$-dimensional Brownian motion, which is needed to prove Theorem \ref{th:main}-(ii). For $\eps>0$, consider the regularized system
\begin{equation}\label{eq:sde_reg}
\begin{cases}
  dV_{t}=\m \big(t,V_{t}\big)dt+\s \big(t,V_{t}\big) dW_t, %& \quad V^{(\eps)}_{0} = v,
   \\
  dX^{\eps}_{t}=\big(V_{t} + F^{(\eps)}(t,X^{\eps}_t) \big) dt, %& \quad X^{(\eps)}_{0} = x.
\end{cases}
\end{equation}
where %the $\m^{(\eps)}$, $\s^{(\eps)}$ and 
$F^{(\eps)}$ is the smooth approximation of %$\m$, $\s$ and 
$F$ defined through \eqref{eq:moll}. We also denote by $\theta^{(\eps)}_{s,t}$, $0\leq t\leq s \leq T$, %the unique 
a forward flow associated to the drift in \eqref{eq:sde_reg}, namely
\begin{align}%\label{eq:}
&\frac{d}{dt}\theta^{(\eps)}_{s,t}(v,x) = \big( \m%^{(\eps)}
\big(s,[\theta^{(\eps)}_{s,t}(v,x) ]_1\big), [\theta^{(\eps)}_{s,t}(v,x) ]_1 + F^{(\eps)}(s,[\theta^{(\eps)}_{s,t}(v,x) ]_2 \big)\big), \\%\qquad 
&\theta^{(\eps)}_{t,t}(v,x) = (v,x),
\end{align}
where for $z\in \R^{2d}$ the notation $[{z}]_i,\ i\in \{1,2\} $ stands for the $d$ first components ($i=1$) or the last $d$ components ($i=2$) of $z$.
Finally, for any $\lambda>0$ we define the $2d$-dimensional Gaussian density
\begin{equation}%\label{eq:}
g_{\lambda}(t,v,x): = \frac{1}{(2\pi \lambda)^{{d}}  t^{2d}} \exp\Big(  -\frac{t^{-1}|v|^2 + t^{-3}|x|^2 }{2 \lambda} \Big), \qquad (t,v,x)\in (0,\infty)\times \R^{2d}.
\end{equation}
The following result directly follows from \cite[Theorem 1.1 and Remark 1.3]{MR4554678}.
\begin{proposition}\label{prop:density_bound}
Let $\alpha = 2$ and let Hypothesis \emph{[{\bf Coe}]} be in force. Assume
\begin{equation}%\label{eq:}
\mu,\sigma \in L^\infty_TC^\g_b, \qquad F \in L^\infty_TC^{\frac{1+\g}{3}}, 
%f\in L^\infty_T({\bf C}_v^\gamma \cap {\bf C}_x^{(1+\g)/(1+\alpha)}),
\end{equation}
for some $\gamma\in(0,1)$. Then the system \eqref{eq:sde_reg} is weakly well-posed for any $\eps\geq 0$. 

Furthermore, for any $\eps>0$, the following upper-bound on its transition density {$p^{(\eps)}%(s,v,x;t,\eta,\xi)
$} holds: 
\begin{align}\label{eq:density_upper_bound_2dim}
p_{(V,{X^\eps})}^{(\eps)}(t,v,x;s,\eta,\xi) \leq C  g_\lambda\big(s-t,  (\eta,\xi) - \theta^{(\eps)}_{s,t}(v,x) \big), \\%\qquad 
0\leq t <s\leq T,\ (v,x), (\eta,\xi)\in\R^{2d},
\end{align}
with $C,\lambda$ positive constants, independent on $\eps$.
In particular, we have
\begin{equation}\label{eq:density_upper_bound_1dim}
p^{(\eps)}_{{X^\eps}}(t,v,x;s,\xi) \leq C\, \Gamma\big( \lambda (s-t)^3, \xi - [ \theta^{(\eps)}_{s,t}(v,x)]_{2} \big), 
\end{equation}
where 
{
\begin{equation}\label{eq:def_marginal}
p^{(\eps)}_{{X^\eps}}(t,v,x;s,\xi) := \int_{\mathbb{R}^d} p^{(\eps)}_{(V,X^{\eps})}(t, v, x; s, \eta, \xi)d\eta,
\end{equation}
and where} $\G$ is the standard $d$-dimensional Gaussian density in \eqref{eq:standard_gauss}.
\end{proposition}
\begin{remark}\label{rem:strong_well_regular}
Under the assumptions of Proposition \ref{prop:density_bound}, for $\eps>0$, pathwise uniqueness for \eqref{eq:sde_reg} also holds as long as it holds for the autonomous component (see Theorem \ref{th:well_posedness_autonomous}). This follows %from Theorem \ref{th:well_posedness_autonomous} and 
by noticing that $F^{(\eps)}(t,x)$ is Lipschitz continuous in $x$, uniformly with respect to $t\in [0,T]$.
\end{remark}
\begin{remark}\label{rem:upper_bounds}
By \cite[Remark 1.3]{MR4554678}, we have that the bounds \eqref{eq:density_upper_bound_2dim}-\eqref{eq:density_upper_bound_1dim} also hold for $\eps =0$, with $\theta^{(0)}_{s,t}$, $0\leq t\leq s \leq T$, being any forward Peano-type flow associated to the drift in \eqref{eq:sde}-\eqref{eq:sde_x}.
\end{remark}

\subsection{Schauder estimates for kinetic-type fractional operators}\label{sec:schauder}

{The results in this section could be proved for $d\geq1$. However, since they will be utilized with $d=1$, we state and prove them here for $d=1$, to promote ease of reading. We} state some regularity estimates for the Cauchy problem 
\begin{equation}\label{eq:cauchy_zvonkin_gen}
\begin{cases}
\big[ \partial_t + \Lc_v + (v + F(t,x))\partial_x - \lambda \big] u(t,v,x) = -f(t,v,x) , \quad (t,v,x)\in(0,T)\times \R^{2},\\%\times \R^d,\\
 u (T,\cdot,\cdot) \equiv 0,
\end{cases}
\end{equation} 
where 
\begin{equation}\label{eq:Lcal}
\Lc_v = \frac{\sigma^{\alpha}(t,v)}{2}\, \Delta^{{\alpha}/{2}}_v + \mu(t,v) \partial_v,
\end{equation}
with $\Delta_v^{{\alpha}/{2}}$ denoting the $1$-dimensional fractional Laplacian acting on the variable $v$, namely
\begin{equation}\label{eq:1d_frac_lap}
\Delta_v^{{\alpha}/{2}} g(v) = 
\begin{cases}
 \int_{\R} \big[ g(v+w) - g(v-w) - 2g(v)  \big] |w|^{-1-\alpha} dw , &\quad \alpha\in (1,2),\\
{\partial_{vv} g(v)}, &\quad \alpha = 2.
\end{cases}
\end{equation}
The estimates in Proposition \ref{prop:Schauder} below stem from the recent results in \cite{MR4124429}. We point out that the latter actually hold for a larger class of PDEs, which include the multi-dimensional version of \eqref{eq:cauchy_zvonkin_gen}, i.e. when $(v,x)\in \R^{2d}$ and $\Delta_v^{{\alpha}/{2}}$ is the $d$-dimensional fractional Laplacian in \eqref{eq:gen_lap_frac}. 

To state and prove Proposition \ref{prop:Schauder}, we introduce the H\"older-Zygmund spaces employed in \cite{MR4124429}. For $h\in\R$, denote by $\delta_h$ the first-order difference operator acting on $g:\R\to\R$ as
\begin{equation}%\label{eq:}
\delta_h g (x) = g(x+h) - g(x),
\end{equation}
and, for a general order $m\in \mathbb{N}$, set
\begin{equation}%\label{eq:}
%\delta^1_h = \delta_h,\qquad \delta_h^{m} = \delta_h \circ \delta^{m-1}_h,\ m\geq 2. 
\delta^m_h = \underbrace{\delta_h \circ \dots \circ \delta_h}_{\text{$m$ times}}.
\end{equation}
\begin{notation}
Let $\beta>0$. We denote by: 
\begin{itemize}
\item ${\bf C}^{\beta}$ the $\beta$-order H\"older-Zygmund space given by the functions $g:\R\to\R$ such that
\begin{equation}%\label{eq:}
\| g \|_{{\bf C}^{\beta}} := \sup_{x\in\R}|g(x)| +  \sup_{h,x\in\R, h\neq 0} \big| \delta^{\lfloor \beta \rfloor +1}_h g (x) \big| / |h|^{\beta} < \infty,
\end{equation}
where $\lfloor \beta \rfloor$ denotes the greatest integer less than $\beta$. 
\item ${\bf C}_v^{\beta}$ and ${\bf C}_x^{\beta}$ the space of the functions $g=g(v,x):\R^2\to\R$ such that
\begin{equation}%\label{eq:}
\| g \|_{{\bf C}_v^{\beta}} := \sup\limits_{x\in\R} \|g(\cdot,x)\|_{{\bf C}^\beta} < \infty \qquad \text{and} \qquad \| g \|_{{\bf C}_x^{\beta}} := \sup\limits_{v\in\R} \|g(v,\cdot)\|_{{\bf C}^\beta} < \infty,
\end{equation}
respectively.
%Let $\nu,\beta >0$. We denote by 
\item $L^\infty_T {\bf C}^\b$ the space of measurable functions $f=f(t,\cdot):[0,T]\to C_b(\R)$ such that
%$L^\infty_TC^\b$ and $L^\infty_TC^\b_b$ denote the spaces of measurable functions $f\colon[0,T]\to C_b$ such that
\begin{equation}
\|f\|_{L^\infty_T {\bf C}^\b} :=	 \sup\limits_{t\in[0,T]} \|f(t,\cdot)\|_{{\bf C}^\beta}  < +\infty.
	\end{equation}
\item $L^\infty_T {\bf C}_v^\beta$ and $L^\infty_T {\bf C}_x^\beta$ the spaces of measurable functions $f=f(t,\cdot,\cdot):[0,T]\to C_b(\R^2)$ such that
%$L^\infty_TC^\b$ and $L^\infty_TC^\b_b$ denote the spaces of measurable functions $f\colon[0,T]\to C_b$ such that
\begin{equation}%\label{eq:}
\| f \|_{L^\infty_T{\bf C}_v^{\beta}} := \sup\limits_{t\in[0,T]} \|f(t,\cdot,\cdot)\|_{{\bf C}_v^\beta} < \infty %\qquad \text{and} \qquad \| f \|_{L^\infty_T{\bf C}_x^{\beta}} := \sup\limits_{t\in[0,T]} \|f(t,\cdot,\cdot)\|_{{\bf C}_x^\beta} < \infty,
\end{equation}
and
\begin{equation}
\| f \|_{L^\infty_T{\bf C}_x^{\beta}} := \sup\limits_{t\in[0,T]} \|f(t,\cdot,\cdot)\|_{{\bf C}_x^\beta} < \infty,
\end{equation}
	respectively.
%\item $L^\infty_T({\bf C}_v^\nu \cap {\bf C}_x^\b)$ the space of measurable functions $f=f(t,\cdot,\cdot):[0,T]\to C_b(\R^2)$ such that
%%$L^\infty_TC^\b$ and $L^\infty_TC^\b_b$ denote the spaces of measurable functions $f\colon[0,T]\to C_b$ such that
%\begin{equation}
%\|f\|_{L^\infty_T({\bf C}_v^\nu \cap {\bf C}_x^\b)} :=	 \sup\limits_{x\in\R} \|f(\cdot,\cdot,x)\|_{L^{\infty}_T {\bf C}^\nu} \ +  \sup\limits_{v\in \R} \|f(\cdot,v,\cdot)\|_{L^{\infty}_T{\bf C}^\b} < +\infty.
%	\end{equation}
	\end{itemize}
%	respectively.
\end{notation}

\begin{proposition}[Schauder estimates]\label{prop:Schauder}
Let $\alpha\in(1,2]$ and $\lambda\geq0$. Let also Hypothesis \emph{[{\bf Coe}]} be in force and assume
\begin{equation}%\label{eq:}
\mu,\sigma \in L^\infty_TC^\g_b, \qquad F \in L^\infty_TC^{\frac{1+\g}{1+\alpha}}, \qquad 
%f\in L^\infty_T({\bf C}_v^\gamma \cap {\bf C}_x^{(1+\g)/(1+\alpha)}),
f\in L^\infty_T {\bf C}_v^\gamma \cap L^\infty_T {\bf C}_x^{(1+\g)/(1+\alpha)} ,
\end{equation}
for some $\gamma\in(0,1)$. Then, there is a classical solution $u$ to \eqref{eq:cauchy_zvonkin_gen}, namely a continuous function $u=u(t,v,x)$ such that 
\begin{equation}%\label{eq:}
u(t,v,x) = \int_t^T \big[ \Lc_v u + (v + F(t,x))\partial_x u - \lambda  u + f  \big] (s,v,x) ds, \quad (t,v,x)\in [0,T]\times\R^2,
\end{equation}
for which the following estimate holds:
\begin{equation}\label{eq:schauder}
\|\partial_v u \|_{L^\infty_T{\bf C}_v^{\alpha+\gamma-1} } + 
\|\partial_v u \|_{L^\infty_T {\bf C}_x^{(\alpha+\gamma-{\varepsilon})/(\alpha+1)}} \leq C\big( \| f \|_{L^\infty_T{\bf C}_v^\gamma }  +  \| f \|_{L^\infty_T {\bf C}_x^{(1+\g)/(1+\alpha)}} \big) , %\qquad \| \partial_v u  \|_{L^\infty_T([0,T]\times \R^2)}\leq \frac{1}{2}, 
\end{equation}
%\begin{equation}\label{eq:schauder}
%\|\partial_v u \|_{L^\infty_T({\bf C}_v^{\alpha+\gamma-1} \cap {\bf C}_x^{(\alpha+\gamma-\epsilon)/(\alpha+1)})} \leq C \| f \|_{L^\infty_T({\bf C}_v^\gamma \cap {\bf C}_x^{(1+\g)/(1+\alpha)})}  , %\qquad \| \partial_v u  \|_{L^\infty_T([0,T]\times \R^2)}\leq \frac{1}{2}, 
%\end{equation}
for {any $\eps>0$} and for some $C>0$ {independent of $\lambda$}. %and $\eps>0$ small enough. 
Furthermore, if $\lambda>0$ is large enough we have
\begin{equation}\label{eq:bound_grad_u}
%\|  u  \|_{L^\infty_T([0,T]\times \R^2)} +
 \| \partial_v u  \|_{L^\infty_T([0,T]\times \R^2)} + \| \partial_x u  \|_{L^\infty_T([0,T]\times \R^2)} 
\leq \frac{1}{2} .
\end{equation}
\end{proposition}
{Observe that the complete Schauder estimate, which is reported below in \eqref{eq:estimate_zhang_6.11}, actually gives uniqueness of the classical solution satisfying the estimate.}
\begin{remark}
From Proposition \ref{prop:Schauder} we can recover analogous Schauder estimates for the non-kinetic Cauchy problem
\begin{equation}\label{eq:cauchy_zvonkin_gen_uni}
\begin{cases}
\big[ \partial_t + \Lc_v  - \lambda \big] u(t,v) = -f(t,v) , \qquad (t,v)\in(0,T)\times \R,\\%\times \R^d,\\
 u_{\lambda} (T,\cdot) \equiv 0,
\end{cases}
\end{equation} 
which are useful to prove pathwise uniqueness of the autonomous equation %in the system 
\eqref{eq:sde}. Indeed, from Proposition \ref{prop:Schauder}, it readily follows that there exists a continuous function $u=u(t,v)$ such that 
\begin{equation}%\label{eq:}
u(t,v) = \int_t^T \big[ \Lc_v u  - \lambda  u + f  \big] (s,v) ds, \qquad (t,v)\in [0,T]\times\R,
\end{equation}
for which the following estimate holds:
\begin{equation}
\|\partial_v u \|_{L^\infty_T{\bf C}^{\alpha+\gamma-1} } \leq C \| f \|_{L^\infty_T {\bf C}^\gamma }, %\qquad \| \partial_v u  \|_{L^\infty_T([0,T]\times \R^2)}\leq \frac{1}{2}, 
\end{equation}
for some $C>0$ (independent of $\lambda$). Furthermore, if $\lambda>0$ is large enough we have
\begin{equation}\label{eq:bound_grad_auton}
%\|  u  \|_{L^\infty_T([0,T]\times \R)} + 
 \| \partial_v u  \|_{L^\infty_T([0,T]\times \R)}\leq \frac{1}{2} .
\end{equation} 
\end{remark} 
To prove Proposition \ref{prop:Schauder}, we introduce the mixed H\"older-Zygmund spaces employed in \cite{MR4124429}. For $h\in\R$ and $m\in \mathbb{N}$, define the operators $\delta_{h,1}, \delta_{h,2}$ acting on $g=g(v,x):\R^2 \to \R$ as
\begin{equation}%\label{eq:}
\delta^m_{h,1} g(v,x) = \delta^m_h g(\cdot,x)(v), \qquad \delta^m_{h,2} g(v,x) = \delta^m_h g(v,\cdot)(x),
\end{equation}
and the mixed norm %\red{(invertire ${\bf C}_x^{\nu}  {\bf C}_v^{\beta}$ con $  {\bf C}_v^{\beta} {\bf C}_x^{\nu}$ visto che stiamo usando sempre $(v,x)$ e non $(x,v)$?)}
\begin{equation}%\label{eq:}
\|g\|_{{\bf C}_x^{\nu}  {\bf C}_v^{\beta}} :=\|g\|_{ {\bf C}_v^{\beta}} + \|g\|_{{\bf C}_x^{\nu} }+ \sup_{(v,x),(h,h')\in \R^2}\frac{ \big|  \delta^{\lfloor \nu \rfloor + 1}_{h,2}\delta^{\lfloor \beta \rfloor + 1}_{h',1} g(v,x)  \big| }{|h|^{\nu} |h'|^{\beta}}, \qquad \nu,\beta>0.
\end{equation} 
Finally, we set $L^{\infty}_T{\bf C}_x^{\nu}  {\bf C}_v^{\beta}$ as the set of functions $f=f(t,\cdot,\cdot):[0,T] \to C_b(\R^2)$ such that
\begin{equation}%\label{eq:}
\| f \|_{L^{\infty}_T{\bf C}_x^{\nu}  {\bf C}_v^{\beta}} : =  \sup_{t\in[0,T]}  \|f(t,\cdot,\cdot)\|_{{\bf C}_x^{\nu}  {\bf C}_v^{\beta}} <\infty.    
\end{equation}
\begin{remark}\label{rem:inequ_mixed}
A direct computation shows that, for any $\delta,\nu>0$ we have
\begin{equation}%\label{eq:}
\|\partial_v f\|_{L^{\infty}_T {\bf C}_x^{\nu}} \leq \| f \|_{L^\infty_T{\bf C}_x^{\nu}  {\bf C}_v^{1+\delta}  }.
\end{equation}
\end{remark}
{We will need the following estimate, which is proved in \cite[Lemma 4.1-(i)]{MR4124429}.}
%The following interpolation lemma is a particular case of \cite[Lemma 4.1-(i)]{MR4124429}.
\begin{lemma}\label{lemm:interpol}
Let $\beta_1,\beta_2,\nu_1,\nu_2 \geq 0$, $\theta\in (0,1)$. There exists $C>0$ (depending on all the previous indices) such that  
\begin{equation}%\label{eq:}
\|g\|_{{\bf C}_x^{\nu}  {\bf C}_v^{\beta}} \leq C \|g\|^\theta_{{\bf C}_x^{\nu_1}  {\bf C}_v^{\beta_1}} \, \|g\|^{1-\theta}_{{\bf C}_x^{\nu_2}  {\bf C}_v^{\beta_2}},
\end{equation}
where
\begin{equation}%\label{eq:}
\beta = \theta \beta_1 + (1-\theta) \beta_2 ,\qquad \nu = \theta \nu_1 + (1-\theta) \nu_2 .
\end{equation}
\end{lemma}
The following interpolation lemma is a particular case of \cite[Theorem 2.2 - Corollary 2.3]{MR4124429}.
\begin{lemma}\label{lemm:interpol_bis}
For any $0\leq s< r< t$ there is $C>0$ (depending on the previous parameters), such that 
\begin{equation}%\label{eq:}
\| g \|_{{\bf C}_v^{r}} \leq   \d   \| g \|_{{\bf C}_v^{t}}     +     C \d^{\frac{s-r}{t-r}}  \| g \|_{{\bf C}_v^{s}},
\qquad \delta\in (0,1).
\end{equation}
\end{lemma}

\begin{proof}[Proof of Proposition \ref{prop:Schauder}]
\emph{Case $\alpha\in(1,2)$.} \cite[Theorem 6.11]{MR4124429} states that a classical solution $u=u(t,v,x)$ exists for which the following estimates hold:
\begin{align}
\| u \|_{L^\infty_T{\bf C}_v^{\alpha+\gamma}} &+ \| u \|_{L^\infty_T {\bf C}_x^{(\alpha+\gamma+1)/(\alpha+1)} } 
+ \| u \|_{L^\infty_T{\bf C}_x^{(1-\rho)(1+\gamma)/(\alpha+1)}  {\bf C}_v^{\alpha+\rho \gamma}  } \\ \label{eq:estimate_zhang_6.11}
& \leq C\big( \| f \|_{L^\infty_T{\bf C}_v^\gamma }  +  \| f \|_{L^\infty_T {\bf C}_x^{(1+\g)/(1+\alpha)}} \big),\\ 
\|  u  \|_{L^\infty_T([0,T]\times \R^2)}&\leq \lambda^{-1} {\|  f  \|_{L^\infty_T([0,T]\times \R^2)}},  \label{eq:estimate_zhang_6.11_bis}
\end{align} 
for some $\rho\in(0,1)$ small enough and some $C>0$ independent of $\lambda$.
Estimate \eqref{eq:estimate_zhang_6.11} directly yields 
\begin{equation}\label{eq:schauder_first}
\|\partial_v u \|_{L^{\infty}_T {\bf C}_v^{\alpha+\gamma-1}} \leq C\big( \| f \|_{L^\infty_T{\bf C}_v^\gamma }  +  \| f \|_{L^\infty_T {\bf C}_x^{(1+\g)/(1+\alpha)}} \big).
\end{equation}
%\begin{equation}\label{eq:schauder_first}
% \sup\limits_{x\in\R} \|\partial_v u (\cdot,\cdot,x)\|_{L^{\infty}_T {\bf C}^{\alpha+\gamma-1}} \leq C  \| f \|_{L^\infty_T({\bf C}_v^\gamma \cap {\bf C}_x^{(1+\g)/(1+\alpha)})}.
%\end{equation}
Furthermore, applying the {interpolation Lemma \ref{lemm:interpol}} with 
\begin{align}%\label{eq:}
 ({\nu_1},\beta_1) &= \big( (1-\rho)(1+\gamma)/(\alpha+1) , \alpha + \rho \gamma \big), \quad ({\nu_2},\beta_2) = \big( (\alpha+\gamma+1)/(\alpha+1) , 0 \big), \\%\quad 
 \theta &= \frac{1+\delta}{\alpha + \rho \gamma},
\end{align}
 with $\delta>0$ suitably small, {\eqref{eq:estimate_zhang_6.11}} yields 
 \begin{equation}%\label{eq:}
\| u \|_{L^\infty_T{\bf C}_x^{(\alpha+\gamma-\eps)/(\alpha+1)}  {\bf C}_v^{1+\delta}  }
\leq C\big( \| f \|_{L^\infty_T{\bf C}_v^\gamma }  +  \| f \|_{L^\infty_T {\bf C}_x^{(1+\g)/(1+\alpha)}} \big), \quad \eps =\delta + (1+\delta) \frac{\rho}{\alpha + \rho \gamma},
\end{equation}
which in turn, by Remark \ref{rem:inequ_mixed}, implies
\begin{equation}%\label{eq:}
 \|\partial_v u\|_{L^{\infty}_T {\bf C}_x^{(\alpha+\gamma-\eps)/(\alpha+1)}} \leq C\big( \| f \|_{L^\infty_T{\bf C}_v^\gamma }  +  \| f \|_{L^\infty_T {\bf C}_x^{(1+\g)/(1+\alpha)}} \big).
\end{equation}
%\begin{equation}%\label{eq:}
% \sup\limits_{v\in\R} \|\partial_v u (\cdot,v,\cdot)\|_{L^{\infty}_T {\bf C}lpha+\gamma-\eps)/(\alpha+1)}} \leq C  \| f \|_{L^\infty_T({\bf C}_v^\gamma \cap {\bf C}_x^{(1+\g)/(1+\alpha)})}.
%\end{equation}
This, together with \eqref{eq:schauder_first}, proves \eqref{eq:schauder}.

Finally, by the interpolation {Lemma \ref{lemm:interpol_bis}} with $s=0$, $r=1+\gamma$ and $t=\alpha+\gamma$, we obtain
\begin{equation}\label{eq:estimate_grad1}
\| u \|_{L^\infty_T{\bf C}_v^{1+\gamma}} \leq \delta \| u \|_{L^\infty_T{\bf C}_v^{\alpha+\gamma}} + \kappa\, \delta^{-\frac{1+\gamma}{\alpha-1}} \|  u  \|_{L^\infty_T([0,T]\times \R^2)}, \qquad \delta \in (0,1).
\end{equation}
Also, applying the same estimate with $s=0$, $r=({\alpha+1+\frac{\gamma}{2}})/({\alpha+1})$ and $t=({\alpha+1+{\gamma}})/({\alpha+1})$ yields
\begin{equation}\label{eq:estimate_grad2}
\| u \|_{L^\infty_T{\bf C}_x^{({\alpha+1+\frac{\gamma}{2}})/({\alpha+1})}}  \leq \delta \| u \|_{L^\infty_T{\bf C}_x^{({\alpha+1+{\gamma}})/({\alpha+1})}} + \kappa\, \delta^{-\frac{\alpha+1+\gamma/2}{\gamma/2}} \|  u  \|_{L^\infty_T([0,T]\times \R^2)}, \quad \delta \in (0,1).
\end{equation}
{By employing first \eqref{eq:estimate_grad1}-\eqref{eq:estimate_grad2} with $\delta$ small enough, then \eqref{eq:estimate_zhang_6.11}, and finally \eqref{eq:estimate_zhang_6.11_bis} with $\lambda$ large enough, we obtain \eqref{eq:bound_grad_u}}. 

\vspace{2pt}

\emph{Case $\alpha=2$.} The proof is identical to that of the case $\alpha\in(1,2)$, provided that \eqref{eq:estimate_zhang_6.11}-\eqref{eq:estimate_zhang_6.11_bis} hold true for $\alpha=2$. This can be checked by adapting the proof of \cite[Theorem 6.11]{MR4124429}. This is a straightforward exercise, and we omit the details for brevity.
\end{proof}

%We consider the Cauchy problem
%\begin{equation}\label{eq:cauchy_app}
%\begin{cases}
%\(\Lc_{v} + (v + F(t,x))\p_x +\p_t\) u(t,v,x)= G(t,x) + \l u(t,v,x),\\
%u(T,\cdot)=0,
%\end{cases}
%\end{equation}
%where $G\in L^\infty_TC^{\b/(1+\a)}_b$.
%\begin{theorem}[{\cite[Theorem 6.3]{MR4124429}}]\label{th:zhang}
%Let $\a\in]1,2]$. %and $\b\in]0,1[,\vartheta\in]0,\beta\wedge(\a-1)],\gamma\in[\b,1+\a[$. 
%Under \emph{[{\bf H-Coe}]},\emph{[\bf H-$\sigma$-$\gamma$]} and \emph{[{\bf H-$\mu F$-$\gamma$}]}, for any $T > 0$, 
%there exists a classical solution $u_\l$ of \eqref{eq:cauchy_app}. Moreover, 
%there exists $\e=\e(\a,\b,\g)>0$ small enough and $C=C(\a,\b,\g) > 0$ such that 
%for any $\l\geq0$ and any classical solution $u$ of \eqref{eq:cauchy_app},
%	\begin{equation}
%	\|u_\l\|_{L^\infty_T(C+\g)}_v \cap C+\g)/(\a+1)}_x)} +
%	\|u_\l\|_{L^\infty_T(C+\e\g)}_v C^{(1-\e)\g/(\a+1)}_x )}
%	\leq C \| G \|_{L^\infty_TC^{\b/(1+\a)}_b} \,.
%	\end{equation}
%Finally, there exists $\l\geq0$ such that
%	\begin{equation}
%	\|u_{\l}\|_{L^\infty_T(C^1_v \cap C^1_x)}\leq \frac{1}{2}.
%	\end{equation}
%\end{theorem}
%From this theorem, via interpolation, one can obtain the following
%\begin{corollary}\label{cor:schauder}
%Under the same assumptions of Theorem \ref{th:zhang}, we have that 
%	\begin{equation}
%	\|u_\l\|_{L^\infty_T(C^{1}_v C+\g-\e)/(\a+1)}_x )}
%	\leq C \| G \|_{L^\infty_TC^{\b/(1+\a)}_b}
%	\end{equation}
%\end{corollary}
%
%\begin{proof}[Proof of Corollary \ref{cor:schauder}]
%INTERPOLATION...
%\end{proof}

%%%%%%%%%%%%%%%%%%%
%		ONE DIMENSION		%
%%%%%%%%%%%%%%%%%%%
%

\section{One-dimensional setting
}\label{sec:proof_th_one}

Throughout this section set $d=1$, fix $\alpha\in(1,2]$ and consider $\mu,\sigma,F:[0,T]\times{\R\to \R}$  coefficient functions satisfying the assumptions of Theorem \ref{th:main}-(i). We need to prove strong well-posedness of Eq. \eqref{eq:sde_x} under the assumption that strong well-posedness holds for Eq. \eqref{eq:sde}. 

\subsection{Preliminary estimates}\label{sec:regul_estimates}

Fix $\lambda>0$ sufficiently large and consider, for any $m\in\mathbb{N}$, the Cauchy problem
\begin{equation}\label{eq:cauchy_app0}
\begin{cases}
\big[\Lc_{v} + (v + F(t,x))\p_x+\p_t -\l \big] u_m= - F_m(t,x),\\
u_m(T,\cdot)=0.
\end{cases}
\end{equation}
Here $\Lc_{v}$ is defined as in \eqref{eq:Lcal}, $F_m(t,x):=F(t,x) \chi_m(x)$, and $\chi_m$ is a smooth cutoff function on $\R$ such that
	\begin{equation}\label{eq:cutoff}
		\chi_m(x)=1, \  |x| \leq m, \qquad \chi_m(x)=0, \  |x|>m+1.
	\end{equation}
%and such that $ \| \chi_m \|_{ {\bf C}_x^{(1+\g)/(1+\alpha)}}$ is independent of $m$.
%Then, %for $\lambda>1$ sufficiently large, 
{Note that we here introduce a cut-off in order to apply the Schauder estimates of Proposition \ref{prop:Schauder} which are given for a bounded source term}.
By Proposition \ref{prop:Schauder} with $f=F_m$, a suitable $\eps>0$ and $\lambda$ large enough, there exists a classical solution $u_m$ to \eqref{eq:cauchy_app0} such that
%$C>0$ %, \red{$\a\leq\b < \a+\g$} 
%and $\l>0$ such that, fixing $\b=\frac{\a+1}{2}<\a$,
	\begin{align}
	| u_m(t,v,x)- u_m(t,v',x')| & \leq  \frac{1}{2}\big( |v-v'| + |x-x'| \big), \label{eq:lip_u}\\
	|\p_v u_m(t,v,x)-\p_v u_m(t,v,x')| & \leq C |x-x'|^{\frac{1}{2}}, \label{eq:hold_u}
	\end{align}
for any $ t \in [0, T]$ and $v,x,x'\in\R$, with $C>0$ positive constant. %independent of $m$.
\begin{remark}
We stress that, as we derived \eqref{eq:schauder}-\eqref{eq:bound_grad_u} in Proposition \ref{prop:Schauder} as a consequence of the general Schauder estimate \eqref{eq:estimate_zhang_6.11}, the constants $\lambda$ and $C$ above seem to depend on $\| F_m\|_{L^{\infty}([0,T]\times \R)}$, and thus on $m$. Actually, as \eqref{eq:lip_u}-\eqref{eq:hold_u} only contain estimates on the differences of $u_m$, these estimates could be obtained with $\lambda$ and $C$ independent of $m$. However, in the proof below, the dependence of $\lambda$ and $C$ on $m$ is not relevant, thus we do not need a sharp result in this sense.
\end{remark}
Now set
\begin{align}%\label{eq:}
\hspace{-20pt}
\phi_m(t,v,x,x')&:= x - x' +  u_m(t,v,x) - u_m(t,v,x'), 	\\
\hspace{-20pt}\uuu_m(t,v,x,x';w)&:=	 u_m(t,v+\s(t,v)w,x)-u_m(t,v,x)\\&\quad\;\,						-u_m(t,v+\s(t,v)w,x')+u_m(t,v,x'), \label{eq:def_h_m}
\end{align}
for any $t\in [0,T]$ and $v,w,x,x'\in\R$. We have the following

\begin{lemma}\label{lem:zm}
Let %$\t_m$ and 
$\phi_m$ and $\uuu_m$ be defined as above. Then there exists $\kappa>0$ such that, {for any $m\in\Nb$},%we have
\begin{align}
 \label{eq:zm3}
|\uuu_m(t,v,x,x';w)| & \leq \kappa |w| %\times |\s(t,v)|\times 
(1 \wedge |x-x'|^{1/2} ),\\  
|x-x'| & \leq 2 (|\phi_m(t,v,x,x')|\wedge |\phi_m(t,v,x,x') + \uuu_m(t,v,x,x';w)|),\\\label{eq:zm3_bis}
\end{align}
for any $t\in[0,T]$, and $v,w,x,x'\in\R$.%, with $\kappa>0$ independent of $m$.
%\begin{align}
%|\uuu_m(s,V_{s\wedge\t_m},X_{s\wedge\t_m},X'_{s\wedge\t_m};v)|\leq& C |v| |\s(s,V_{s\wedge\t_m})|,
%\label{eq:zm3}\\
%|\uuu_m(s,V_{s\wedge\t_m},X_{s\wedge\t_m},X'_{s\wedge\t_m};v)|\leq& C |v| \ |\s(s,V_{s\wedge\t_m})| |X_{s\wedge\t_m}-X'_{s\wedge\t_m}|^{\b/(\a+1)}.
%\label{eq:zm4}
%\end{align}
%
\end{lemma}
%\begin{remark} \label{rem:Z^m0}
\begin{proof}
%We have that
%\begin{align}
%\uuu_m(s,V_{s\wedge\t_m},X_{s\wedge\t_m},X'_{s\wedge\t_m};v)	
%				=	&u_m(s,V_{s\wedge\t_m}+v\s(s,V_{s\wedge\t_m}),X_{s\wedge\t_m})-u_m(s,V_{s\wedge\t_m},X_{s\wedge\t_m}) \\
%					&-u_m(s,V_{s\wedge\t_m}+v\s(s,V_{s\wedge\t_m}),X'_{s\wedge\t_m}) + u_m(s,V_{s\wedge\t_m},X'_{s\wedge\t_m}).
%\end{align}
%On the one hand, employing the Lipschitz continuity on $u$ we immediately have \eqref{eq:zm3}.
{Noticing that $\p_{\rho}u_m$ is bounded (see \eqref{eq:lip_u})}, by the fundamental theorem of calculus, we have				
\begin{equation}			
\uuu_m(t,v,x,x';w)			=	\int_0^{w\s(t,v)} \big( \p_{\rho}u_m(t,v+\rho,x) -  \p_{\rho}u_m(t,v+\rho,x') \big) d\rho .
				\end{equation}
				{Thus \eqref{eq:zm3} stems from \eqref{eq:lip_u}-\eqref{eq:hold_u} and because $\s$ is bounded.}
				
By triangular inequality and by the Lipschitz estimate \eqref{eq:lip_u}, we have
\begin{equation}
|\phi_m(t,v,x,x')| \geq  |x - x'| - |u_m(t,v,x)-u_m(t,v,x')| \geq  \frac{1}{2}|x - x'|.
\end{equation}
Similarly,
\begin{align}%\label{eq:}
&|\phi_m(t,v,x,x') + \uuu_m(t,v,x,x';w)| \\& = |x-x' + u_m(t,v+\s(t,v)w,x) - u_m(t,v+\s(t,v)w,x')| \\
& \geq |x - x'| - |u_m(t,v+{\s(t,v)w},x)-u_m(t,v{+\s(t,v)w},x')| \\
&\geq  \frac{1}{2}|x - x'|,
\end{align}
and thus we obtain \eqref{eq:zm3_bis}.
\end{proof}

\subsection{Proof of Theorem \ref{th:main}-(i)}\label{sec:proof_th_one_core}
% for $\alpha=2$, which we re-write here for the reader's convenience:
% \begin{align}\label{eq:sde_diff_V}
%%\begin{cases}
%  dV_{t}&=\m(t,V_{t})dt+\s(t,V_{t})dW_{t},\\
%  dX_{t}&=(V_{t} + F(t,X_t) ) dt.	\label{eq:sde_diff_X}
%%\end{cases}
%\end{align}
%consider a regularized version $u^{\eps}$ defined as follows. 
%For any $\eps>0$ we set $F_{\eps}:\R^d \to \R^d$ as 
%\begin{equation}%\label{eq:}
%F_{\eps} (x) :=  \Phi_{\eps} * F (x) = \int_{\R^d} \Phi_{\eps}(x-y) F(y) dy, \qquad \Phi_{\eps}(y) = \frac{1}{\eps^d} \Phi(y/\eps),
%\end{equation}
%with $\Phi$ being a standard mollifier, and l

By Proposition \ref{prop:density_bound}, the system \eqref{eq:sde}-\eqref{eq:sde_x} is weakly solvable. Therefore, by {Yamada-Watanabe theorem} {(see \cite[Theorem 1.1, Chapter IV]{ikeda2014stochastic} 
%{\red for the case $\alpha=2$, and \cite[Theorem 14.94]{MR0542115} for the case $\alpha\in(1,2)$})}, 
{for the case $\alpha=2$, and \cite[Theorem 137]{MR2160585} for the case $\alpha\in(1,2)$})}, 
we only have to show pathwise uniqueness. To this end, we fix hereafter: 
\begin{itemize}
\item A filtered probability space $(\Omega, \F, (\F_t)_{t\in[0,T]}, \Pb)$ satisfying the usual assumptions of completeness and right-continuity on the filtration, and a scalar symmetric (i.e. rotationally invariant) $\alpha$-stable L\'evy process {$L^{(\alpha)}$ adapted to $(\F_t)_{t\in[0,T]}$}, namely a Markov process generated by the fractional Laplacian $\Delta_v^{\alpha/2}$ in \eqref{eq:1d_frac_lap};
\item A solution $(V_t)_{t\in[0,T]}$ to \eqref{eq:sde}, % Note that, by Theorem \ref{th:well_posedness_autonomous}, 
which is, by assumption, pathwise unique up to fixing the initial condition $V_0$, measurable with respect to $\F_0$; %Also, it is not restrictive to assume $V$ surely continuous;
\item Two solutions $(X_t)_{t\in[0,T]}$ and $(X'_t)_{t\in[0,T]}$ to \eqref{eq:sde_x}, {defined on $(\Omega, \F, (\F_t)_{t\in[0,T]}, \Pb)$}, such that $X_0 = X'_0$ $\Pb$-almost surely.
\end{itemize}
We need to prove
\begin{equation}\label{eq:X_X'_equal_bis}
X_t = X'_{t} \quad \Pb\text{-a.s.}, \qquad t\in (0,T].
\end{equation}

\paragraph{First step: Zvonkin transform}

%Let $(V_t,X_t)$ be a solution of equation \eqref{eq:sde}-\eqref{eq:sde_x}. Since $u$ is bounded, b
Let $m\in\Nb$. %$t\in [0,T]$, 
By It\^o formula we have%, for any $m\in\Nb$, 
\begin{align}\label{eq:ito_um}
du_m(t,V_t,X_t)= 	&\(\Lc_{v} + (v + F(t,x))\p_x + \p_t\)u_m(t,V_{t-},X_{t})dt  + d M^X_t \\
%			&+ \int_\R \(u_m(t,V_{t-}+\s(t-,V_{t-})v,X_{t-})-u_m(t,V_{t-},X_{t-})\) \tilde J (dt, dv)\\
		=	&  \(- F_m (t,X_{t})+ \l u_m(t,V_{t-},X_{t})\) dt + d M^X_t ,%\\
%			&+ \int_\R \(u_m(t,V_{t-}+\s(t-,V_{t-})v,X_{t-})-u_m(t,V_{t-},X_{t-})\) \tilde J (dt, dv)
\end{align}
where we have used that $u_m$ solves \eqref{eq:cauchy_app0}, and where $M^X$ is a local martingale to be later specified.
Same formulas hold for $du_m(t,V_t,X'_t)$. 
Therefore, %as $X$ is a solution to \eqref{eq:sde_x}, we obtain %we can give a representation for the integral of the $F$ along $X$, namely
setting the stopping time
\begin{equation}\label{eq:tau_m}
		\t_m \coloneqq \inf\{s > 0 \colon |X_s| \vee |X'_s| \vee |\Delta L^{(\alpha)}_s| \geq m\},% \red{\ \text{or} \ |\Delta S^{\a}_t|>m}\}
	\end{equation}
	and owing to \eqref{eq:cutoff}, we obtain
\begin{align}%\hspace{-30pt}
\int_0^{t\wedge\t_m} F (s,X_{s}) ds &= \int_0^{t\wedge\t_m} F_m (s,X_{s}) ds \\
			& = %\int_0^t(1-\chi_m(X_s))F(s,X_s) ds + 
		u(0,V_0,X_0) - u_m({t\wedge\t_m},V_{t\wedge\t_m},X_{t\wedge\t_m}) \\&+ \l \int_0^{t\wedge\t_m} u_m (s,V_{s},X_{s}) ds + M^X_{t\wedge\t_m}, %\\
%			&+ \int_0^t \int_\R \(u(s,V_{s-}+\s(s-,V_{s-})v,X_{s-})-u(s,V_{s-},X_{s-})\) \tilde J (ds, dv). 
\label{eq:F}
\end{align}
%where $\RRm_s(x) = F (s,x)-F_m (s,x)=(1-\chi_m(x))F(s,x)$.
and the same representation holds by replacing $X$ with $X'$. 
%Let $(V,X)$ and $(V,X')$ be two solutions of \eqref{eq:sde}-\eqref{eq:sde_x} defined on the same probability space and with the same starting point. 
%We set $\t_m$ the following stopping time: 
%	
%notice that 
%	\begin{align}
%		\int_0^{t\wedge\t_m}  (1-\chi_m(X_s))F(s,X_s) ds = 0
%	\end{align}
%\begin{align}
%\int_0^{t\wedge\t_m} F (s,X_{s}) ds 
%		=	& \int_0^{t\wedge\t_m}  (1-\chi_m(X_s))F(s,X_s) ds + u(0,V_0,X_0) - u(t,V_{t\wedge\t_m} ,X_{t\wedge\t_m} ) + \int_0^{t\wedge\t_m}  \l u (s,V_{s},X_{s}) ds \\
%			&+ \int_0^{t\wedge\t_m}  \int_\R \(u(s,V_{s-}+\s(s-,V_{s-})v,X_{s-})-u(s,V_{s-},X_{s-})\) \tilde J (ds, dv)\\
%		=	& u(0,V_0,X_0) - u(t,V_{t\wedge\t_m} ,X_{t\wedge\t_m} ) + \int_0^{t\wedge\t_m}  \l u (s,V_{s},X_{s}) ds \\
%			&+ \int_0^{t\wedge\t_m}  \int_\R \(u(s,V_{s-}+\s(s-,V_{s-})v,X_{s-})-u(s,V_{s-},X_{s-})\) \tilde J (ds, dv)
%\end{align}
%since $\chi_m(X_s)=1$ for any $s<\t_m$; the same holds true if we change $X'$ with $X$. 
%Then, exploiting the representation for $F$ in \eqref{eq:F}, we get 
Thus, as both $X$ and $X'$ are solutions to \eqref{eq:sde_x}, we obtain
\begin{align}
X_{t\wedge\t_m} - X'_{t\wedge\t_m} 	& = \int_0^{t\wedge\t_m} F (s,X_{s}) ds - \int_0^{t\wedge\t_m} F (s,X'_{s}) ds	\\
			& =  u_m(t\wedge\t_m,V_{t\wedge\t_m},X'_{t\wedge\t_m}) - u_m(t\wedge\t_m,V_{t\wedge\t_m},X_{t\wedge\t_m}) \\
			&\quad + \l\int_0^{t\wedge\t_m}  \big(u_m(s,V_{s},X_{s})-u_m(s,V_{s},X'_{s})\big) ds + M^X_{t\wedge\t_m}-M^{X'}_{t\wedge\t_m} .%M^X_{t\wedge\t_m} - M^{X'}_{t\wedge\t_m},
		% \\
%			&+ \int_0^{t\wedge\t_m} \int_\R \(u_m(s,V_{s-}+\s(s-,V_{s-})v,X_{s-})-u_m(s,V_{s-},X_{s-})\) \tilde J (ds, dv)\\
%			&- \int_0^{t\wedge\t_m} \int_\R \(u_m(s,V_{s-}+\s(s-,V_{s-})v,X'_{s-})-u_m(s,V_{s-},X'_{s-})\) \tilde J (ds, dv). %\\
%		=	&  u_m(t,V_{t\wedge\t_m},X'_{t\wedge\t_m}) - u_m(t,V_{t\wedge\t_m},X_{t\wedge\t_m}) \\
%			&+  \l\int_0^{t\wedge\t_m}  \(u_m(s,V_{s},X_{s})-u_m(s,V_{s},X'_{s})\) ds \\
%			&+ \int_0^{t\wedge\t_m} \int_\R \uuu_m(s,V_{s-},X_{s-},X'_{s-};v) \tilde J (ds, dv),
\label{eq:diff_X_Xprime}
\end{align}
%\end{remark}

\paragraph{Second step: Watanabe-Yamada technique}

We start with the following technical lemma. Although it is a classical result (see \cite[Chapter IV (Theorem 3.2)]{ikeda2014stochastic}), we provide a proof for completeness.
\begin{lemma}\label{lem:watanabe}
There exists a sequence $(\f_n)_{n\in\Nb}$ of functions in $C^2(\R)$ {and a decreasing sequence of positive real numbers $(a_n)_{n\in\Nb}$} such that, for any $x\in\R$ and $n\in\Nb$,
\begin{itemize}
\item[(i)] $\f_n(x)\geq 0$, and $\f_n(0)=0$,
\item[(ii)] $|\f'_n(x)|\leq1$, 
\item[(iii)] $|\f''_n(x)|\leq \frac{1}{n}\frac{1}{|x|} \mathds{1}_{]a_n,a_{n-1}[}(x)$, %where $a_n=e^{-\frac{n(n+1)}{2}}$ 
\item[(iv)] $\f_n(x)\nearrow|x|$ as ${n\to\infty}$.
\end{itemize}
\end{lemma}
\begin{proof}
%See \cite[Proof of Theorem 3.2 (p. 182)]{ikeda_watanabe}.
We set $a_n=e^{-n(n+1)}$ for $n\in\Nb$. %$a_0=1$ and 
Then
\begin{equation}\label{eq:int_an}
\int_{a_n}^{a_{n-1}} \frac{1}{2 w}  dw = n.%, \qquad n\geq1
\end{equation}
For $n\in\Nb$, let $\psi_n$ be a continuous function supported on $[a_n,a_{n-1}]$ such that
\begin{equation}
0 \leq \psi_n(x) \leq \frac{1}{ n x}, \qquad \int_{a_n}^{a_{n-1}} \psi_n(y) dy =1.
\end{equation}
Note that such a $\psi_n$ exists thanks to \eqref{eq:int_an}. 
We define 
\begin{equation}
\f_n(x) \coloneqq \int_{0}^{|x|} \int_{0}^{\y} \psi_n(y) dy d\y.
\end{equation}
One can easily check that all properties $(i)$-$(iv)$ are satisfied by $\f_n$. {In particular, $(iv)$ stems from 
\begin{equation}%\label{eq:}
\f_n(0) =0, \qquad \f'_n(x) \leq \f'_{n+1}(x)\quad\text{if } x\geq 0,    \qquad \f'_n(x) \geq \f'_{n+1}(x)\quad\text{if } x\leq 0,
\end{equation}
for any $x\in\R$ and $n\in\Nb$.}
\end{proof}

For any $m\in\Nb$ and $t\in [0,T]$, set
%Then we set
\begin{align}\label{eq:Zm_def_bis} 
Z^m_t& := \phi_m(t\wedge\t_m, V_{t\wedge\t_m}, X_{t\wedge\t_m}, X'_{t\wedge\t_m}) \\
& = X_{t\wedge\t_m} - X'_{t\wedge\t_m} + u_m(t\wedge\t_m,V_{t\wedge\t_m},X_{t\wedge\t_m})-u_m({t\wedge\t_m},V_{t\wedge\t_m},X'_{t\wedge\t_m}), \label{eq:Zm_def} %\qquad t\geq 0,
\end{align}
and observe that \eqref{eq:zm3_bis} yields
\begin{equation}
|X_{t\wedge\t_m} - X'_{t\wedge\t_m}|\leq 2 |Z^m_{t}|.
\label{eq:zm1}
\end{equation}
%For, by %definition of $Z^m$ and 
%triangular inequality and by the Lipschitz estimate \eqref{eq:lip_u}, we have
%\begin{equation}
%|Z^m_{t}| \geq  |X_{t\wedge\t_m} - X'_{t\wedge\t_m}| - |u_m(t\wedge\t_m,V_{t\wedge\t_m},X_{t\wedge\t_m})-u_m(t\wedge\t_m,V_{t\wedge\t_m},X'_{t\wedge\t_m})| \geq  \frac{1}{2}|X_{t\wedge\t_m} - X'_{t\wedge\t_m}|.
%\end{equation}
%where we employed the Lipschitz estimate \eqref{eq:lip_u} in the second inequality. 

We now need the following estimate, whose proof is postponed until Section \ref{sec:proof_bound_phi_n}.

\begin{proposition}\label{prop:bound_phi_Z}
Let $(\f_n)_{n\in\Nb}$ be a sequence as in Lemma \ref{lem:watanabe}. For any $m\in\Nb$, there exists $C_m>0$ such that
\begin{equation}\label{eq:bound_phi_n_Zm}
E\[\f_n(Z^m_t)\] \leq \l \int_0^t E[|Z^m_s|] ds + C_m \frac{t}{n}, \qquad t\in[0,T],\quad n\in\Nb.
\end{equation}
\end{proposition}	
Now, by Lemma \ref{lem:watanabe}-{(iv)}, %(ii)
we obtain $E\[\f_n(Z^m_t)\] \to E\[|Z^m_t|\]$ as $n\to \infty$. Thus
%\begin{align}
%|E\[\f_n(Z_t)\]|\xrightarrow[]{n\to\infty}E\[|Z_t|\]\leq |A_1|\leq 2\l \int_0^t E[|Z_s|] ds
%\end{align}
\begin{equation}%\label{eq:}
E\[|Z^m_t|\]\leq \l \int_0^t E[|Z^m_s|] ds, \qquad t\in[0,T],
\end{equation}
and Gr\"onwall lemma yields $E\[|Z^m_t|\]=0$. 
%\begin{align}
%E\[|Z^m_t|\]=0.
%\end{align}
Owing %once more 
to estimate \eqref{eq:zm1}, we obtain 
\begin{equation}
E\[|X_{t\wedge\t_m}-X'_{t\wedge\t_m}|\]  \leq 2 E\[|Z^m_t|\] =0.
\end{equation}
Finally, by letting $m\to+\infty$, it follows that $E\[|X_t-X'_t|\]=0$ for any $t\in[0,T]$. This yields \eqref{eq:X_X'_equal_bis} and concludes the proof of Theorem \ref{th:main}-(i).
\hfill\qedsymbol{}

\subsection{Proof of Proposition \ref{prop:bound_phi_Z}}\label{sec:proof_bound_phi_n}\label{sec:proof_th_one_prop}
Denoting by $\tilde J (ds, dw)$ the compensated Poisson measure of the process $L^{(\alpha)}$ {when $\alpha\in (1,2) $}, the process $M^X$ in \eqref{eq:ito_um} reads as
\begin{equation}%\label{eq:}
M^X_t = 
\begin{cases}
\int_0^t \int_\R \big(u_m(s,V_{s-}+\s(s,V_{s-})w,X_{s})-u_m(s,V_{s-},X_{s})\big) \tilde J (ds, dw)& \\&\hspace{-50pt}\text{if } \alpha\in(1,2),\vspace{5pt} \\ 
\int_0^t  \sigma(s,V_{s})\p_{v}u_m(s,V_{s},X_{s})dW_s & \hspace{-50pt} \text{if } \alpha=2,
\end{cases}
\end{equation}
and the same representation holds for $M^{X'}$. Therefore, by \eqref{eq:diff_X_Xprime}-\eqref{eq:Zm_def}, $Z^m_t$ can be represented as
%Therefore $Z^m_t$ is a semimartingale for every $m\in\Nb$, and
\begin{equation}
%dZ^m_t=	&\l\(u_m(t,V_{t},X_{t})-u_m(t,V_{t},X'_{t})\)dt \\
%		&+ \uuu_m(t,V_{t-},X_{t-},X'_{t-};v)\tilde J (dt, dv),\\
Z^m_t=	\l\int_0^{t\wedge\t_m}\big(u_m(s,V_{s},X_{s})-u_m(s,V_{s},X'_{s})\big)ds  +M_t,\label{DIN_ZM}
\end{equation}
with
\begin{equation}%\label{eq:}
M_t: = 
\begin{cases}
\int_0^{t\wedge\t_m} \int_{|w|\leq m} \uuu_m(s,V_{s-},X_{s},X'_{s};w)\tilde J (ds, dw) &  \text{if } \alpha\in(1,2),\vspace{5pt} \\ 
\int_0^{t\wedge\t_m}  \sigma(s,V_{s})\big( \p_{v}u_m(s,V_{s},X_{s}) -  \p_{v}u_m(s,V_{s},X'_{s}) \big) dW_s &  \text{if } \alpha=2,
\end{cases}
\end{equation}
%with
%\begin{equation}%\label{eq:}
%M_t = M^X_t - M^{X'}_t = 
%\begin{cases}
%\int_0^t \int_\R \uuu_m(s-,V_{s-},X_{s-},X'_{s-};w) \tilde J (ds, dw)& \quad \text{if } \alpha\in(1,2),\vspace{5pt} \\ 
%\int_0^t \sigma(s,V_{s}) \big( \p_{v}u_m(s,V_{s},X_{s}) - \p_{v}u_m(s,V_{s},X'_{s}) \big) dW_s & \quad \text{if } \alpha=2,
%\end{cases}
%\end{equation}
and where $\uuu_m(s,v,x,x';w)$ is as defined in \eqref{eq:def_h_m}. {Note that the jumps of $M_t$ are bounded by the definition of $\tau_m$ in \eqref{eq:tau_m}.}
%\begin{equation}
%\uuu_m(s,v,x,x';w):=	 u_m(s,v+\s(s,v)w,x)-u_m(s,v,x) 						-u_m(s,v+\s(s,v)w,x')+u_m(s,v,x').
%\end{equation}
%We now prove a technical estimate that will be used in the rest of the proof, and rely on the regularity properties of $u_m$.
%\eqref{eq:zm1} and \eqref{eq:zm2} set a comparison between $Z_m$ and $X-X'$, 
%while \eqref{eq:zm3} and \eqref{eq:zm4} are estimates for $\uuu_m$. 
%All these estimates strongly rely on the regularity properties of the solution to \eqref{eq:cauchy_app0} $u_m$.

%Notice that, since $\uuu_m$ and $\f_n'$ are bounded, the integral 
%\begin{equation}
%\int_\R \big(\f_n(Z^m_{s}+\uuu_m(s,V_{s},X_{s},X'_{s};v))-\f_n(Z^m_{s})-\f_n'(Z^m_{s})\uuu_m(s,V_{s},X_{s},X'_{s};v)\big) \v(dv),
%\end{equation}
%is bounded (independently from $Z^m$). 

Let $(\f_n)_{n\in\Nb}$ be a sequence approximating the absolute value as in Lemma \ref{lem:watanabe}. Fix $n,m\in\Nb$ and $t\in [0,T]$.
%Then we can apply It\^o formula to $\f_n(Z^m_t )$ and we get, for any $n,m\in\Nb$,
Recalling that $|w|^{-\alpha-1} dw$ is the L\'evy measure of $L^{(\alpha)}$, by {It\^o's formula and \eqref{DIN_ZM}}, we obtain
\begin{equation}%\label{eq:}
\f_n(Z^m_t ) = I^1_t + I^2_t + I^3_t, \qquad t\in[0,T],
\end{equation}
with %\red{(non sarebbe il caso di specificare la misura di Levy e da dove viene $|w|^{-1-\alpha}\, dw$??)}
\begin{align}
		I^1_t	& :=  \l\int_0^{t\wedge\t_m}  \f_n'({Z^m_{s}})\big(u_m(s,V_{s},X_{s})-u_m(s,V_{s},X'_{s})\big) ds, \\
		I^2_t	& := 
		\begin{cases}
		\int_0^{t\wedge\t_m}\! \int_{|w|\leq m} \big(\f_n(Z^m_{s}+\uuu_m(s;w))-\f_n(Z^m_{s})-\f_n'(Z^m_{s})\uuu_m(s;w)\big)|w|^{-1-\alpha}\, dw\, ds &\\&\hspace{-100pt}\text{if } \alpha\in(1,2),\vspace{5pt} \\ 
\frac{1}{2}\int_0^{t\wedge\t_m} \f''_n( Z^m_s) \sigma^2(s,V_{s})\big( \p_{v}u_m(s,V_{s},X_{s}) -  \p_{v}u_m(s,V_{s},X'_{s}) \big)^2 ds & \\&\hspace{-100pt}\text{if } \alpha=2,
		\end{cases}\hspace{-10pt}
		\\
	I^3_t		&:=
	\begin{cases}
	 \int_0^{t\wedge\t_m} \int_{|w|\leq m} \big(\f_n(Z^m_{s-}+\uuu_m(s,V_{s-},X_s,X'_s;w))-\f_n(Z^m_{s-})\big) \tilde J (ds, dw)& \\
	 &\hspace{-50pt}\text{if } \alpha\in(1,2),\vspace{5pt} \\ 
\int_0^{t\wedge\t_m} \f'_n( Z^m_s) \sigma(s,V_{s})\big( \p_{v}u_m(s,V_{s},X_{s}) -  \p_{v}u_m(s,V_{s},X'_{s}) \big) dW_s	& \\&\hspace{-50pt}\text{if } \alpha=2, 
%			\\
%		=	&  \l\int_0^{t}  \f_n'(Z^m_s)\(u_m(s,V_{s\wedge\t_m},X_{s\wedge\t_m})-u_m(s,V_{s\wedge\t_m},X'_{s\wedge\t_m})\)\mathds{1}_{s\leq\t_m} ds \\
%			&+ \int_0^{t} \int_\R \big(\f_n(Z^m_{s}+\uuu_m(s,V_{s\wedge\t_m},X_{s\wedge\t_m},X'_{s\wedge\t_m};v))-\f_n(Z^m_{s\wedge\t_m})-\f_n'(Z^m_{s\wedge\t_m})
%			\times\\
%			& \qquad\qquad\times\uuu_m(s,V_{s\wedge\t_m},X_{s\wedge\t_m},X'_{s\wedge\t_m};v)\big) \mathds{1}_{s\leq\t_m}\v(dv)\ ds\\
%			&+ \int_0^{t\wedge\t_m} \int_\R \big(\f_n(Z^m_{s}+\uuu_m(s,V_{s},X_{s},X'_{s};v))-\f_n(Z^m_{s})\big) \tilde J (ds, dv).
\end{cases}
\end{align} 
where we set $\uuu_m(s;w):=\uuu_m(s,V_{s},X_s,X'_s;w)$ {in $I_t^2$} in order to shorten notation. 
%Note that the integral on $\R$ appearing in the term $I^2_t$ (case $\alpha\in(1,2)$) is convergent, because of estimate \eqref{eq:zm3} and for ${\blue\varphi'_n}$ is bounded (Lemma \ref{lem:watanabe}-(ii)) (see the computations below).
{The integrability of the integrand in $I^2_t$ will be checked below, where the bound for $E[|I^2_t|]$ is proved.}

To prove \eqref{eq:bound_phi_n_Zm}, we need to bound $E[|I^j_t|]$, {$j=1,2$, and show $E[I^3_t]=0$}.

\paragraph{Bound for {$E[|I^1_t|]$}} We have, by triangular inequality,
\begin{align}
|I^1_t|		&\leq  \l\int_0^{t\wedge\t_m}  |\f_n'(Z^m_s)|\ |u_m(s,V_{s},X_{s})-u_m(s,V_{s},X'_{s})|ds \\
&\leq
\l\int_0^t  |\f_n'(Z^m_s)|\ |u_m(s,V_{s},X_{s\wedge\t_m})-u_m(s,V_{s},X'_{s\wedge\t_m})|ds 
\intertext{(by Lemma \ref{lem:watanabe}-(ii) and by the Lipschitz estimates \eqref{eq:lip_u})}
	\leq	& \frac{\l}{2} \int_0^t  |X_{s\wedge\t_m}-X'_{s\wedge\t_m}| ds.
%	\leq	  {\l}  \int_0^t  E\[|Z^m_{s}|\] ds ,
\end{align}
Applying estimate \eqref{eq:zm1} and passing to the expected value yields
\begin{equation}%\label{eq:}
E[|I^1_t|] \leq {\l}  \int_0^t  E\[|Z^m_{s}|\] ds.
\end{equation}

\paragraph{Bound for {$E[|I^2_t|]$}} %${}$
%\vspace{2pt}
%
%\noindent\underline{\emph{Case $\alpha\in(1,2)$:}}
First note that, by applying second order Taylor formula, we can write
\begin{equation}
\f_n\big(z+h \big)-\f_n(z)
-\f_n'(z)h = \frac{1}{2} \f_n''(\xi)h^2, \qquad z,h\in\R,
\end{equation}
where $\xi$ is a point between $z$ and $z+h$. Therefore, Lemma \ref{lem:watanabe}-(iii) and \eqref{eq:zm3_bis} yield
\begin{align}
\hspace{-30pt}&\big| \f_n\big(\phi_m(s,v,x,x')+\uuu_m(s,v,x,x'{;w}) \big)-\f_n\big(\phi_m(t,v,x,x')\big) \\&-\f_n'\big(\phi_m(t,v,x,x')\big)\uuu_m(s,v,x,x';w) \big|\\
 &\leq \frac{{|\uuu_m(s,v,x,x';w)|^2}}{n |x-x'|} \leq \kappa \frac{|w|^2}{n}, \label{eq:estim_tay}
\end{align}
for any $s\in [0,T]$ and $v,w,x,x'\in\R$, with the last inequality stemming from \eqref{eq:zm3}.

Now, if $\alpha\in(1,2)$, we have, by triangular inequality,
\begin{align}
|I^2_t|		&\leq \int_0^{t}\! \int_{|w|\leq m} \big|\f_n(Z^m_{s}+\uuu_m(s;w))-\f_n(Z^m_{s})-\f_n'(Z^m_{s})\uuu_m(s;w)\big| |w|^{-1-\alpha}\, dw\, ds
\intertext{(by the definition of $Z^m$ in \eqref{eq:Zm_def_bis} together with \eqref{eq:estim_tay})}
%		&\leq \int_0^t \int_\R  \big|\f_n\big(Z^m_{s\wedge\t_m}+\uuu_m(s\wedge\t_m;w)\big)-\f_n(Z^m_{s\wedge\t_m})-\f_n'(Z^m_{s\wedge\t_m})\uuu_m(s\wedge\t_m;w)\big| \times |w|^{-1-\alpha}\, ds \\
		&\leq \frac{\kappa}{n} \int_0^t \int_{|w|\leq m}  |w|^{1-\alpha} \, dw\, ds .
\end{align}
%where we employed the estimate for \eqref{eq:estim_tay}. 
Therefore, we proved
\begin{equation}\label{eq:est_val_att_I2}
E[|I^2_t|] \leq \frac{C_m t}{n}
\end{equation}
in the case $\alpha\in (1,2)$. For $\alpha=2$, estimate \eqref{eq:est_val_att_I2} readily follows from Lemma \ref{lem:watanabe}-(iii), together with \eqref{eq:zm1} and \eqref{eq:hold_u}, and because $\sigma$ is bounded.
%\vspace{2pt}
%
%\noindent\underline{\emph{Case $\alpha=2$:}}
\begin{remark}
{We could here have developed at order 2 up to a given fixed \textit{macro} threshold and at order 1 beyond, separating as usual the small and large jumps and balancing both contributions in $n$ and $m$ (keeping in mind that we want a control which vanishes when $n$ goes to infinity). Actually, for our purpose, the previous computations are sufficient.}
\end{remark}

\paragraph{Null expectation of {$I^3_t$}} By Lemma \ref{lem:watanabe}-(ii), together with estimate \eqref{eq:zm3}, we have

\begin{align}
	&E\[\int_0^{T\wedge \tau_m} \int_{|w|\leq m} \big(\f_n(Z^m_{s}+\uuu_m(s,V_{s},X_{s},X'_{s};w))-\f_n(Z^m_{s})\big)^2 |w|^{-1-\alpha}\, dw\,ds\] 
	\\
	&\leq \kappa\, T \int_{|w|\leq m}|w|^{1-\alpha}\, dw , 
%	\\
%\leq	&E\[\int_0^t \int_\R |\uuu_m(s,V_{s\wedge\t_m},X_{s\wedge\t_m},X'_{s\wedge\t_m};v))|^2 \nu(dv)\,ds\]\\
%=	&E\[\int_0^t \int_\R  |\uuu_m(s,V_{s\wedge\t_m},X_{s\wedge\t_m},X'_{s\wedge\t_m};v))|^2 \mathds{1}_{|v|\leq1} \nu(dv)\,ds\]\\
%	&+E\[\int_0^t \int_\R  |\uuu_m(s,V_{s\wedge\t_m},X_{s\wedge\t_m},X'_{s\wedge\t_m};v))|^2 \mathds{1}_{|v|>1} \nu(dv)\,ds\]\\
%	&=A_1+A_2.
	\end{align}
and, since $\sigma$ and $u_m$ are bounded, we also have
\begin{equation}%\label{eq:}
E\[ \int_0^{T\wedge\t_m} \Big(\f'_n( Z^m_s) \sigma(s,V_{s})\big( \p_{v}u_m(s,V_{s},X_{s}) -  \p_{v}u_m(s,V_{s},X'_{s}) \big)\Big)^2 ds \] <\infty.
\end{equation}
%is a square integrable martingale with zero mean. 
%
%in fact, since $\s$ and $\p_v u$ are bounded, 
%$$\int_0^{t} \f_n'(Z^m_s) \s(s,V_s)(\p_v u_m(s,V_s,X_s)-\p_v u_m(s,V_s,X'_s)) dW_s$$
Therefore, $I^3_t$ is a square-integrable zero-mean martingale {for $\a\in(1,2]$},  { i.e. $E[I^3_t] = 0$}.

This concludes the proof of \eqref{eq:bound_phi_n_Zm}. 
\hfill\qedsymbol{}

%\begin{remark}
%For the estimates in Lemma \ref{lem:zm} to work, you that $\s$ does not depend on $X$. 
%Adding a dependence to the $x$ variable in $\s$ 
%makes the comparison between $Z_m$ and $X-X'$ not sufficient to complete our proof of the uniqueness of the SDE \eqref{eq:sde}-\eqref{eq:sde_x}. 
%\end{remark}

\section{Multi-dimensional diffusive setting
}\label{sec:multi_dim}

The main goal of this section is proving Theorem \ref{th:main}-(ii). 
%\begin{equation}\label{eq:sde_diff}
%\begin{cases}
%  dV_{t}=\m(t,V_{t})dt+\s(t,V_{t})dW_{t},\\
%  dX_{t}=(V_{t} + F(t,X_t) ) dt,	
%\end{cases}
%\end{equation}
Before dwelling into its proof, we discuss some examples of drift functions $F$ that either meet or fail to meet Hypothesis [{\bf Pea}]. 

\subsection{Peano-like models}\label{sec:peano}\label{sec:multi_dim_prelim}
Consider the drift power-function
\begin{equation}\label{eq:power_func}
F(t,x) : = a(t) |x|^{\beta}, \qquad (t,x)\in [0,T]\times \R^d,
\end{equation}
which gives rise to the Peano-like model
\begin{equation}\label{eq:sde_pea}
\begin{cases}
  dV_{t}=\m(t,V_{t})dt+\s(t,V_{t})d W_t,\\
  dX_{t}=(V_{t} + a(t) |X_t|^{\beta} ) dt.	
\end{cases}
\end{equation}
Note that this system was employed in \cite[Section 4]{MR3808994}, precisely with $V_t=W_t$, to build a counter-example for weak-uniqueness when $\beta<1/3$ {(see also \cite{MR4591369} for a an extension of this counter-example to the pure-jump stable case)}. {We can also refer to \cite{MR4358660} in the Brownian setting for a generalized model in which the Brownian noise propagates through a chain of differential equation and \cite{MR4591369} for the stable counterpart}.
\begin{proposition}\label{prop:peano}
Let $d\geq1$, $\beta \in (1/3,1]$, $a\in L^{\infty}([0,T])$. Then $F$ in \eqref{eq:power_func}  
%\begin{equation}%\label{eq:}
%F(t,x) : = a(t) |x|^{\beta}, \qquad (t,x)\in [0,T]\times \R^d.
%\end{equation}
%Then $F$ 
satisfies %condition \eqref{eq:moment_estimate} in 
Hypothesis $[{\bf Pea}]$. 
\end{proposition}
\begin{proof}
%\red{(ho cambiato $\alpha$ con ${\blue\beta}$ per evitare confusione con l'indice di stabilità...)}
%with $g(r) \sim r^{\frac{3}{2}(\beta -1)}$. 
It is enough to prove that
\begin{equation}\label{eq:ineq_peano1}
\int_{\R^d}|y|^{{\beta}-1} \Gamma((s-t)^3,y-\theta) dy \lesssim  (s-t)^{\frac{3}{2}({\beta} -1)}, \qquad t<s,\ \theta\in\R^d.%, \ {\blue\beta}>0.
\end{equation}
%for some $C>0$, for 
%for any ${\blue\beta}>0$. 
Indeed, %a direct computation yields
recalling $F^{(\eps)}(s,y)=\int_{\R^d} F(s,y-z) \Phi_\varepsilon(z) dz$, and denoting by $\nabla F (s,\cdot)$ the gradient of $F (s,\cdot)$, %by changing variables in the integral 
we obtain
\begin{equation}
|\nabla F^{(\eps)}(s,{y})|  \leq  \|a\|_\infty \int_{\R^d} |{y} -z|^{\beta-1} \frac{1}{\eps^d} \Phi(z/\eps) dz  = %\eps^{\beta-1} 
{\|a\|_{L^\infty}} \int_{\R^d} |y - \xi |^{\beta-1} \Phi(\xi) d\xi,
\end{equation}
for any %$\eps>0$ and 
$y\in\R^d$. Now, the last integral above is a continuous function of $y$ and
\begin{equation}%\label{eq:}
 \int_{\R^d} |y - \xi |^{\beta-1} \Phi(\xi) d\xi \leq (|y|-r)^{\beta-1}\lesssim |y|^{\beta - 1}, \qquad |y|>>{2 r},
\end{equation}
where $r>0$ is such that the support of $\Phi$ is contained in the Euclidean ball with radius $r$.
This yields 
\begin{equation}\label{eq:estimate_der_Feps}
|\nabla_y F^{(\eps)}(s,y)| \lesssim |y|^{\b-1} %+ |y|^{\b-\delta-1}
, \qquad (s,y)\in [0,T]\times \R^d,\quad \eps>0.
\end{equation}
This, together with \eqref{eq:ineq_peano1}, implies \eqref{eq:moment_estimate} with $g(s-t) \sim {\lambda^{\frac{\beta-1}{2}}} (s-t)^{\frac{3}{2}{(\beta -1)}}$.

To show \eqref{eq:ineq_peano1}, 
%Let us consider, for $0\leq t < s \leq T$ and $\theta\in\R^d$,
we decompose the left-hand side therein as
\begin{equation}\label{eq:peano_decomp}
%\\ &\int_{\R^{d}}|\nabla F(s,y)|\Gamma((s-t)^3,y-\theta)dy\\
%& \int_{\R^d}|y|^{\b-1} \Gamma((s-t)^3,y-\theta) dy\\
			     =  \underbrace{\int_{|y|\geq(s-t)^{\frac 3 2}} |y|^{{\beta}-1} \Gamma\big((s-t)^3,y-\theta\big) dy}_{=:I_1%(t,\theta;s) 
			     }
			     	+   \underbrace{ \int_{|y|< (s-t)^{\frac 3 2}} |y|^{{\beta}-1} \Gamma\big((s-t)^3,y-\theta\big) dy}_{=:I_2%(t,\theta;s) 
				}.
\end{equation}
First, we note that 
\begin{equation}
I_1\leq  \int_{|y|\geq(s-t)^{\frac 3 2}} (s-t)^{\frac 3 2 ({\beta}-1)} \Gamma\big((s-t)^3,y-\theta\big) dy\leq (s-t)^{\frac 3 2 ({\beta}-1)},
\end{equation}
for $\Gamma((s-t)^3,\cdot-\theta)$ is a probability density. %for any $0<t<s$ and $\theta\in\R^d$. 
On the other hand
\begin{align}
I_2&\leq  {(s-t)^{-\frac{3}{2}d}} \int_{|y|< (s-t)^{\frac 3 2}} {|y|^{{\beta}-1}}\exp\left( -\frac{|y-\theta|^2}{2(s-t)^3} \right) dy \\
	      &
	      \leq  {(s-t)^{-\frac{3}{2}d}} \int_{|y|< (s-t)^{\frac 3 2}}{|y|^{{\beta}-1}} dy 
\intertext{(using polar coordinates, with $|y|=r$)}
	      &\lesssim  {(s-t)^{-\frac{3}{2}d}} \int_{0}^{(s-t)^{\frac 3 2}} {r^{d-2+{\beta}}} dr \\&
	      \lesssim   (s-t)^{\frac{3}{2}({\beta}-1)}
%\int_{|y|>(s-t)^{\frac 3 2}} (s-t)^{\frac 3 2 (\b-1)} p_\eps(t,x;s,y) dy=(s-t)^{\frac 3 2 (\b-1)}
.
\end{align}
\end{proof}
\begin{remark}\label{rem_1d_cou}
Note that the condition $\beta>1/3$ is crucial in order for $F$ in \eqref{eq:power_func}  to satisfy [{\bf Pea}]. 
Indeed, more generally than \eqref{eq:peano_decomp}, one could decompose the left-hand side in \eqref{eq:ineq_peano1} as
%we can let us consider the decomposition of t, alternative to \eqref{eq:peano_decomp},
\begin{equation}
%\\ &\int_{\R^{d}}|\nabla F(s,y)|\Gamma((s-t)^3,y-\theta)dy\\
%& \int_{\R^d}|y|^{\b-1} \Gamma((s-t)^3,y-\theta) dy\\
			     =  \underbrace{\int_{|y|\geq(s-t)^\rho} |y|^{{\beta}-1} \Gamma\big((s-t)^3,y-\theta\big) dy}_{=:I_1%(t,\theta;s) 
			     }
			     	+   \underbrace{ \int_{|y|< (s-t)^\rho} |y|^{{\beta}-1} \Gamma\big((s-t)^3,y-\theta\big) dy}_{=:I_2%(t,\theta;s) 		
						},
\end{equation}
with an arbitrary $\rho >0$. Performing the same computations as before, this would lead to
\begin{equation}%\label{eq:}
I_1 \leq (s-t)^{\rho ({\beta}-1)} , \qquad I_2 \lesssim (s-t)^{\rho (d-1+{\beta})-\frac{3}{2}d}.
\end{equation}
Hence, in order for both $I_1$ and $I_2$ be integrable on $[0,T]$, the following inequalities need to hold:
\begin{equation}%\label{eq:}
\frac{\frac{3}{2}d - 1}{d+{\beta}-1} <\rho <\frac{1}{1-{\beta}}.
\end{equation}
However, a direct computation shows that the interval above is non-empty if and only if ${\beta}>1/3$. In that case, $\rho=3/2$ satisfies the inequalities. Therefore, by imposing Hypothesis [{\bf Pea}] on a power-function $F$ we recover the known regularity threshold for weak-uniqueness. 
\end{remark}

{Following }%Building upon 
the proof of the previous proposition, it is easy to build a {class} of functions $F$, with a local behavior like in \eqref{eq:power_func}, for which Hypothesis $[{\bf Pea}]$ remains satisfied. For instance, one could consider the function given by 
\begin{equation}%\label{eq:}
F(t,x) : = \sum_{n\in\mathbb{N}} a_n(t) |x-b_n|^{\beta_n}, \qquad (t,x)\in [0,T]\times \R^d,
\end{equation}
with each $b_n\in\R^d$, {$\beta_n \in [1/3+\d , 1]$ for $\d>0$} suitably small, and with
\begin{equation}%\label{eq:}
\sum_{n\in\mathbb{N}} |a_n| \in L^{\infty}([0,T]).%, \qquad b_1, \dots, b_n, \dots \in \R^d.
\end{equation}
On the other hand, by letting infinite singularities of the gradient (with equal magnitude) accumulate not too fast around a given point, even in dimension one it is possible to build a function $F$ that is $\beta$-H\"older continuous and almost everywhere differentiable, but with non-summable $\nabla F$%, which implies that Hypothesis [{\bf Pea}] fails to hold
. Consider $d=1$ and define
%\begin{equation}\label{eq:counter_ex_peano}
%F(x)=  \sum_{n\in\Nb} \big| x - \frac{a_n+a_{n+1}}{2} \big|^{\beta} \mathds{1}_{[a_n,a_{n+1}[}(x), \qquad x\in \R,
%\end{equation}
\begin{equation}\label{eq:counter_ex_peano}
F(x)=\sum_{n\in\Nb}F_n(x) \mathds{1}_{[a_n,a_{n+1})}(x), \qquad x\in\R,
\end{equation}
with 
\begin{equation}\label{eq:counter_ex_peano_2}
F_n(x)=
\begin{cases}
(x-a_n)^\b & \text{if } x \in [a_n,\frac{a_n+a_{n+1}}{2}) ,\\
(a_{n+1}-x)^\b & \text{if } x \in [\frac{a_n+a_{n+1}}{2},a_{n+1}) ,
\end{cases}
\end{equation}
where $(a_n)_{n\in\Nb}$ is a strictly increasing sequence in $\R$. Next we show that if the sequence $a_n$ converges not too fast, for instance if $a_{n+1} - a_{n} \sim n^{-1/\beta}$, {then} Hypothesis [{\bf Pea}] is violated. 
%$[0,1]$ such that $\lim\limits_{n\to\infty}a_n=1$ and $a_0=0$.

\begin{figure}[htp]
\centering
\includegraphics[width=0.7\textwidth]{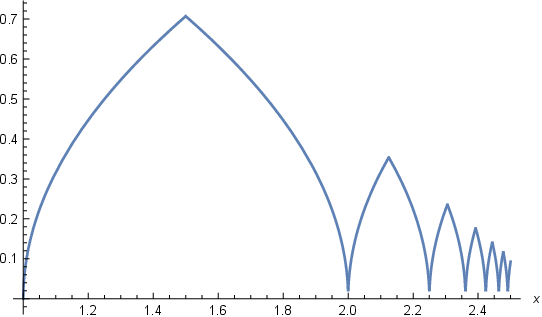}
\caption{Plot of $F(x)$ as in \eqref{eq:counter_ex_peano}-\eqref{eq:counter_ex_peano_2}, with $a_1=1$, $a_{n+1} - a_{n} = n^{-1/\beta}$ and $\beta=1/2$.} 
%\\
% \fbox{confirm: $T=0.5$ here?} \fbox{``sen'' to ``sin''}
\label{fig:plot}
\end{figure}

%\blue{[Aggiungere disegno!]}
%
%\begin{tikzpicture}[xscale=0.5, yscale=2]
%  \draw[->] (-0.5, 0) -- (16, 0) node[right] {$x$};
%  \draw[->] (0, -0.1) -- (0, 2) node[above] {$y$};
%  \draw[scale=1, domain=0:4, smooth, variable=\x]  plot ({\x}, {\x^(0.5)});
%  \draw[scale=1, domain=4:8, smooth, variable=\x]  plot ({\x}, {(8-\x)^(0.5)});
%  \draw[scale=1, domain=8:10, smooth, variable=\x]  plot ({\x}, {(\x-8)^(0.5)});
%  \draw[scale=1, domain=10:12, smooth, variable=\x]  plot ({\x}, {(12-\x)^(0.5)});
%  \draw[scale=1, domain=12:13, smooth, variable=\x]  plot ({\x}, {(\x-12)^(0.5)});
%  \draw[scale=1, domain=13:14, smooth, variable=\x]  plot ({\x}, {(14-\x)^(0.5)});
%  \draw[scale=1, domain=14:14.5, smooth, variable=\x]  plot ({\x}, {(\x-14)^(0.5)});
%  \draw[scale=1, domain=14.5:15, smooth, variable=\x]  plot ({\x}, {(15-\x)^(0.5)});
%  \draw[scale=1, domain=15:15.25, smooth, variable=\x]  plot ({\x}, {(\x-15)^(0.5)});
%  \draw[scale=1, domain=15.25:15.5, smooth, variable=\x]  plot ({\x}, {(15.5-\x)^(0.5)});
%  \node[below] at (15.90,0.5) {$\dots$};
%\end{tikzpicture}
%
%\red{... ho fatto una bozza, ci posso lavorare nei prossimi giorni...}

%\underline{Question:} Is there a function $F\in C^{1/3}$ almost everywhere differentiable such that the estimate \eqref{eq:moment_estimate} is violated in some way?

We have the following
\begin{proposition}[``Counterexample" to Hypothesis $\[{\bf Pea}\]$]\label{prop:counterexample}
For any $\b\in(0,1]$, the function $F$ as defined in \eqref{eq:counter_ex_peano}-\eqref{eq:counter_ex_peano_2} does not satisfy Hypothesis $\[{\bf Pea}\]$ if the sequence $(a_n)_{n\in\Nb}$ is bounded (and thus has a limit) and
\begin{equation}%\label{eq:}
\sum_{n\in\Nb} ({a_{n+1}-a_n})^\b = + \infty.
\end{equation}
%the sequence $(a_n)$ is bounded, and thus has a limit.
\end{proposition}
\begin{proof}
{We recall Notation \ref{not:moll}; since} $ F^{(\eps)'} \to F'$ as $\eps\to 0^+$, almost everywhere, by Fatou's lemma it is enough to show that
\begin{equation}
\int_{\R} |F'(y)| \Gamma(\tau,y-\theta) dy = \infty, \quad \tau\in (0,1], \ \theta\in\R.
\end{equation}
We have
%\begin{equation}\label{eq:counter_ex_peano_der}
%F'(x)= \beta  \sum_{n\in\Nb} \big| x - \frac{a_n+a_{n+1}}{2} \big|^{\beta-1} \mathds{1}_{[a_n,a_{n+1}[}(x), \qquad x\in \R.
%\end{equation}
%\begin{equation}
%F(x)=\sum_{n\in\Nb}F_n(x) \mathds{1}_{[a_n,a_{n+1}[}(x)
%\end{equation}
%where 
%\begin{equation}
%F_n(x)=
%\begin{cases}
%(x-a_n)^\b & \text{if } x \in [a_n,\frac{a_n+a_{n+1}}{2}[ ,\\
%(a_{n+1}-x)^\b & \text{if } x \in [\frac{a_n+a_{n+1}}{2},a_{n+1}[ ,
%\end{cases}
%\end{equation}
%and $(a_n)_{n\in\Nb}$ is a strictly increasing sequence in $[0,1]$ such that $\lim\limits_{n\to\infty}a_n=1$ and $a_0=0$.
%We have that 
\begin{equation}
F'(x)=\sum_{n\in\Nb}F_n'(x) \mathds{1}_{[a_n,a_{n+1})}(x),
\end{equation}
with 
\begin{equation}
F_n'(x)= \b
\begin{cases}
(x-a_n)^{\b-1} & \text{if } x \in (a_n,\frac{a_n+a_{n+1}}{2}),\\
-(a_{n+1}-x)^{\b-1} & \text{if } x \in [\frac{a_n+a_{n+1}}{2},a_{n+1}).
\end{cases}
\end{equation}
Noticing that $F'$ has compact support, we can bound the Gaussian density from below and obtain
%have that, for any $s,t$ and $\theta$ there exist a constant $C=C(s, t, \theta)$
\begin{align}
\int_{\R} |F'(y)| \Gamma(\tau,y-\theta) dy
& \gtrsim   \int_{a_1}^{a_{\infty}} |F'(y)| dy \gtrsim \sum_{n\in\Nb} \int_{a_n}^{\frac{a_n+a_{n+1}}{2}} (y-a_n)^{\b-1} dy \\&\gtrsim \sum_{n\in\Nb} ({a_{n+1}-a_n})^\b = +\infty.
 \end{align}
\end{proof}
% \\
%&= 2 C \b \sum_{n\in\Nb} \int_{a_n}^{\frac{a_n+a_{n+1}}{2}} (y-a_n)^{\b-1} dy\\
% &= 2 C \sum_{n\in\Nb} \left(\frac{a_{n+1}-a_n}{2}\right)^\b \\
% &= 2^{1-\b} C \sum_{n\in\Nb} \left({a_{n+1}-a_n}\right)^\b.
%\end{align}
%We set $\delta_n=a_{n+1}-a_n$ and $\eps_0=\frac{1-\b}{\b}$. 
%We choose $(a_n)_{n\in\Nb}$ such that $\delta_n = C n^{-(1+\eps_0)}=C n^{-1/\b}$ 
%(where $C$ is a constant that only depends on $\b$ such that $\sum_{n\in\Nb}\delta_n=1$). 
%Then
%\begin{equation}
%\sum_{n\in\Nb} \delta_n^\b = C^\b \sum_{n\in\Nb}  n^{-\b(1+\eps_0)} = \infty.
%\end{equation}
%Therefore, for any $\b\in]0,1[$ we can define $F$ that is 
%\begin{itemize}
%\item globally $\b$-H\"older continuous,
%\item bounded,
%\item supported on $[0,1]$,
%\item not absolutely continuous (since $F'$ is not locally integrable),
%\item estimate \eqref{eq:moment_estimate} is not satisfied (for any $0<t<s<T$ and $\theta\in\R$).
%\end{itemize}

%By \cite[Theorem 1.1.]{MR4554678}, we get%and denoting $Z^2=X^{\eps}$
%\begin{align}%\label{}
%\mathcal{E}_\b(x)\leq\int_0^T\int_{\R^d} \frac{1}{|y|^{1-\b}}\exp\left( -\frac{|y-\theta_{s}(x)|^2}{s^3} \right) \frac{dy}{s^{\frac{3}{2}d}} ds
%\end{align}

\begin{remark}
Note that the function $F$ in \eqref{eq:counter_ex_peano}-\eqref{eq:counter_ex_peano_2} is only a counter-example to Hypothesis [{\bf Pea}] and not to strong uniqueness of \eqref{eq:sde}-\eqref{eq:sde_x}. Indeed, $F$ is $\beta$-H\"older continuous and thus fits the assumptions of Theorem \ref{th:main}-(i) if $\beta>1/(1+\alpha)$ {(in particular with $\alpha=2$)}. Therefore, \eqref{eq:sde}-\eqref{eq:sde_x} is strongly well-posed if $\mu$ and $\sigma$ satisfy the remaining assumptions. {Therefore, Proposition \ref{prop:counterexample} shows that [{\bf Pea}] may fail in a model that is strongly well-posedness, and thus highlights a limit of the proof technique.} %rather than the ill-posedness of the model. 

The problem of finding a $\beta$-H\"older continuous function $F$, with $\beta>1/(1+\alpha)$, for which pathwise-uniqueness holds for \eqref{eq:sde} but fails for \eqref{eq:sde_x} remains completely open.
\end{remark}

\subsection{Proof of Theorem \ref{th:main}-(ii)}\label{sec:multi_dim_core}
Throughout this section we let the dimension $d\in\mathbb{N}$ be fixed, as well as {the} coefficient functions  $\sigma:[0,T] \times \R^d \to \mathcal{M}^{d\times d}$ and $\mu,F:[0,T]\times\R^d\to \R^d$ satisfying the assumptions of Theorem \ref{th:main}-(ii). In the diffusive setting (i.e. $\alpha=2$), we need to prove strong well-posedness for Eq. \eqref{eq:sde_x} under the assumption that strong well-posedness holds for Eq. \eqref{eq:sde}. We re-write here the system for the reader's convenience:
 \begin{align}\label{eq:sde_diff_V}
%\begin{cases}
  dV_{t}&=\m(t,V_{t})dt+\s(t,V_{t})dW_{t},\\
  dX_{t}&=(V_{t} + F(t,X_t) ) dt.	\label{eq:sde_diff_X}
%\end{cases}
\end{align}
%consider a regularized version $u^{\eps}$ defined as follows. 
%For any $\eps>0$ we set $F_{\eps}:\R^d \to \R^d$ as 
%\begin{equation}%\label{eq:}
%F_{\eps} (x) :=  \Phi_{\eps} * F (x) = \int_{\R^d} \Phi_{\eps}(x-y) F(y) dy, \qquad \Phi_{\eps}(y) = \frac{1}{\eps^d} \Phi(y/\eps),
%\end{equation}
%with $\Phi$ being a standard mollifier, and l

By Proposition \ref{prop:density_bound}, {system} \eqref{eq:sde_diff_V}-\eqref{eq:sde_diff_X} is weakly solvable. Therefore, by {Yamada-Watanabe theorem} (see \cite[Theorem 1.1, Chapter IV]{ikeda2014stochastic}), we only have to show pathwise uniqueness. To this end, we fix hereafter: 
\begin{itemize}
\item A filtered probability space $(\Omega, \F, (\F_t)_{t\in[0,T]}, \Pb)$ satisfying the usual assumptions of completeness and right-continuity on the filtration, and a $d$-dimensional (uncorrelated) Brownian motion $(W_t)_{t\in[0,T]}$ with respect to $(\F_t)_{t\in[0,T]}$;
\item A solution $(V_t)_{t\in[0,T]}$ to \eqref{eq:sde_diff_V}, which is, by assumption, %Note that, by Theorem \ref{th:well_posedness_autonomous}, $V$ is 
pathwise unique up to fixing the initial condition $V_0$; %Also, it is not restrictive to assume $V$ surely continuous;
\item Two solutions $(X_t)_{t\in[0,T]}$ and $(X'_t)_{t\in[0,T]}$ to \eqref{eq:sde_diff_X}, {defined on $(\Omega, \F, (\F_t)_{t\in[0,T]}, \Pb)$}, such that $X_0 = X'_0$ $\Pb$-almost surely.
\end{itemize}
We need to prove
\begin{equation}\label{eq:X_X'_equal}
X_\tau = X'_{\tau} \quad \Pb\text{-a.s.}, \qquad \tau\in (0,T].
\end{equation}

Fix now $\tau\in (0,T]$ and $\eps>0$. For any $\omega\in\Omega$, %let $u^{\eps}$ be the unique classical solution to 
consider the random transport equation
\begin{equation}\label{eq:transport_random_eps}
\begin{cases}
\partial_t u_{\eps} (t,x) +  \nabla_x u_{\eps} (t,x) \big(  V_t(\omega) +  F^{(\eps)}( t, x )  \big)  = F^{(\eps)}(t, x)     ,  \quad  (t,x)\in(0,\tau)\times\R^d  , \\
u_{\eps}(\tau,\cdot) \equiv 0. 
\end{cases}
\end{equation}
%For a given $\eps$, 
The latter can be viewed as a family of deterministic transport equations, indexed by the trajectories of $V$.
We have the following
\begin{lemma}\label{lemm:classical_sol_transp}
%For any $\tau>0$ $\eps>0$ and $\omega\in\Omega$, \eqref{eq:transport_random_eps} 
Let $\tau\in (0,T]$, $\eps>0$ and $\omega\in\Omega$. The transport equation \eqref{eq:transport_random_eps} has a unique classical solution, namely a continuous function  $u_{\eps}=u_{\eps,\tau,\omega} : [0,\tau] \times \R^d \to \R^d$ such that $\nabla_x u_{\eps}(t,x)$ is continuous on $[0,\tau] \times \R^d$ and 
\begin{equation}%\label{eq:}
u_{\eps} (t,x) = -  \int_{t}^\tau g_{\eps}(s,x) ds %\big[  \langle  V_s(\omega) +  F^{(\eps)}( s, x )  , \nabla u_{\eps} (s,x)  \rangle - F^{(\eps)}(s, x) \big] ds
, \qquad (t,x) \in [0,\tau]\times \R^d,
\end{equation}
with
\begin{equation}\label{eq:def_g}
g_{\eps}(s,x) :=   F^{(\eps)}(s, x) -  \nabla_x u_{\eps} (s,x) \big(  V_s(\omega) +  F^{(\eps)}( s, x )  \big).
\end{equation}
In particular, the function $g_{\eps}(s,x)$ is continuous in the $x$ variable and bounded on $[0,\tau]\times \R^d$.
%we have that
%\begin{equation}%\label{eq:}
%\partial_t u_{\eps} (t,x) = u\big[  \langle  V_t(\omega) +  F^{(\eps)}( t, x )  , \nabla u_{\eps} (t,x)  \rangle - F^{(\eps)}(t, x) \big], \qquad x\in\R^d,
%
\end{lemma}
%\red{(non sarebbe meglio mantenere la notazione del prodotto scalare vettore gradiente?)}
%We postpone the proof to Section \ref{sec:proof_lemmas}.
\begin{proof}
Note that, $F^{(\eps)}(t,\cdot)$ is smooth and has sub-linear growth, uniformly with respect to $t\in[0,T]$, in light of the assumption $F \in L^\infty_TC^{\frac{1+\g}{3}}$. Therefore, the statement is standard under the extra assumption that $F(t,x)$ is continuous in both variables ({see e.g. \cite{evans:book}}%\footnote{\textcolor{blue}{Da S.M. a S.P. e G.: non è la referenza più chiara, rispetto alla formulazione prevalentemente non lineare. Guardo se trovo meglio...}}
). In general, it follows from a straightforward regularization argument in the time variable. We omit the details for brevity.
\end{proof}
\begin{remark}[{It\^o's} formula]\label{rem:Ito}
The function $u_{\eps}$ possesses enough regularity to apply {deterministic It\^o formula, which is the standard chain rule}. Precisely, for any function $\xi:[0,\tau]\to \R^d$ that is continuous and {with} bounded variation, we have
\begin{equation}%\label{eq:}
u_{\eps}\big(s, \xi(s) \big) = u_{\eps}\big(0, \xi(0) \big) + \int_{0}^{s} g_{\eps}\big(t, \xi(t) \big)  d t + \int_{0}^{s} \nabla_x u_{\eps} \big(t,\xi(t)\big)     d \xi(t) , \qquad s\in [0,\tau].
\end{equation}
In particular, owing to the fact that $X(\omega)$ solves the ODE \eqref{eq:sde_diff_X}, and thus is absolutely continuous on $[0,\tau]$, and utilizing $u_{\eps}(\tau, \cdot)\equiv 0$, we have 
%which in turn yields
%\begin{equation}%\label{eq:}
%u_{\eps}\big(0, \xi(0) \big) = -  \int_{0}^{\tau} g\big(t, \xi(t) \big)      d t - \int_{0}^{\tau} \nabla_x u_{\eps} \big(t,\xi(t)\big)     d g(t).
%\end{equation}
\begin{equation}\label{eq:Ito_X}
u_{\eps}\big(0, X_0(\omega) \big) %= -  \int_{0}^{\tau} \big[ g_{\eps}(t, X_t(\omega) )  +  \big(  V_t(\omega) +  F( t, X_t(\omega) )  \big) \nabla_x u_{\eps} (t,X_t(\omega))  \big]  d t ,
=  - {\int_{0}^{\tau}  g_{\eps}(t, X_t(\omega) ) d t }  -  \int_{0}^{\tau}  \nabla_x u_{\eps} (t,X_t(\omega))  \big(  V_t(\omega) +  F( t, X_t(\omega) )  \big) d t ,
\end{equation}
for any $\omega\in\Omega$. 
Obviously, the same formula holds replacing $X$ with $X'$.
\end{remark}
%For any $\omega\in\Omega$ such that $X(\omega)$ is continuous, we have
%Given now any solution $Z=(V,X)$ to \eqref{eq:sde}, we have 
Omitting, for brevity, the explicit dependence on $\omega$ in $X$ and $u_{\eps}$, by \eqref{eq:def_g} we have
\begin{align}%\label{eq:}
\int_0^{\tau} F^{(\eps)}(t,X_t)  dt & = \int_0^{\tau} \Big( g_\eps (t,X_t) +\nabla_x u_{\eps} (t,X_t)  \big({ V_t + F^{(\eps)}( t, X_t ) } \big)   \Big) dt 
\intertext{(by It\^o formula \eqref{eq:Ito_X})} %on $u_{\eps}(t,X_t(\omega))$, see Remark \ref{rem:Ito})}
& = %\int_0^T \Big( \partial_t u_{\eps} (t,X_t)+\langle  F( X_t )+W_t, \nabla u_{\eps} (t,X_t)\rangle \Big) dt + 
\int_0^{\tau} \nabla_x u_{\eps} (t,X_t)   \big( F^{(\eps)}(t, X_t )-F( t, X_t ) \big)   dt - u_{\eps}(0,X_0), \quad \omega\in\Omega.
\end{align}
%for any $\omega \in \Omega$. 
%Given now two solutions $Z=(V,X),Z=(V,X')$ to \eqref{eq:sde}, we have %can modify \eqref{eq:u_key_prop} as follows:
We can perform the same computations replacing $X$ with $X'$. Therefore, owing to the fact that $X_0 = X'_0$ $\Pb$-almost surely, we obtain%, for any $\eps>0$,
\begin{align}\hspace{-25pt}
X_{\tau} - X'_{\tau} &= \int_0^{\tau} F(t,X_t)  dt - \int_0^{\tau} F(t,X'_t) dt \\ %= \int_0^T \langle  F_{\eps}( X_t )-F( X_t ) ,\nabla u_{\eps} (t,X_t)  \rangle dt \\
& = \int_0^{\tau} \nabla_x u_{\eps} (t,X_t)  \big(  F^{(\eps)}(t, X_t )-F(t, X_t ) \big)   dt +  \int_0^{\tau} \big(  F(t, X_t )-F^{(\eps)}(t, X_t ) \big) dt \\
&\quad + \int_0^{\tau}  \nabla_x u_{\eps} (t,X'_t)   \big(  F^{(\eps)}(t, X'_t )-F(t, X'_t ) \big) dt +  \int_0^{\tau} \big(  F(t, X'_t )-F^{(\eps)}(t, X'_t ) \big) dt, \\%\quad 
&\hspace{320pt}\Pb\text{-a.s.}\label{eq:u_key_prop_eps}
\end{align}
%$\Pb$-almost surely. 
Note that the previous computations were performed path-by-path. In order to conclude the argument, we need probabilistic information on $\nabla_x u_{\eps} (t,X_t)$ and $\nabla_x u_{\eps} (t,X'_t)$, which is provided by the following proposition, whose proof is postponed until Section \ref{sec:proof_lemmas}.
\begin{proposition}\label{lemm:estim_transp}
For any $q\geq 0$ and $\tau\in(0,T]$, we have
%there exists a constant $C_{T,q}>0$ (independent of $\eps$) such that
\begin{equation}\label{eq:estimate_gradient_expect}
\Eb\big[  | \nabla_x u_{\eps} (t,X_t) |^q  \big] + \Eb\big[  | \nabla_x u_{\eps} (t,X'_t) |^q  \big] \leq C_{T,q} , \qquad t\in[0,\tau], \quad \eps>0,
\end{equation}
where $C_{T,q}$ is a positive constant (independent of $\eps$, $t$ and $\tau$).
%for any solution $X$ to the second equation in \eqref{eq:sde}.
\end{proposition}
Now, applying H\"older inequality to \eqref{eq:u_key_prop_eps} yields
\begin{align}
\Eb\big[ |X_{\tau} - X'_{\tau} |  \big] &\leq \int_0^{\tau}  \Eb\big[  | \nabla_x u_{\eps} (t,X_t) |^q  \big]^{\frac{1}{q}}  \times  \Eb[ | F^{(\eps)}(t, X_t )-F(t, X_t ) |^p ]^{\frac{1}{p}}  dt \\ 
&\quad + \int_0^{\tau} \Eb\big[  | \nabla_x u_{\eps} (t,X'_t) |^q  \big]^{\frac{1}{q}}  \times  \Eb[ | F^{(\eps)}(t, X'_t )-F(t, X'_t ) |^p ]^{\frac{1}{p}} dt  \\
&\quad + \int_0^{\tau} \Eb[ | F^{(\eps)}(t, X_t )-F(t, X_t ) | ]   dt \\&\quad+ \int_0^{\tau} \Eb[ | F^{(\eps)}(t, X'_t )-F(t, X'_t ) | ]  dt , \label{eq:bound_XXprime}
\end{align}
for any $p,q>1$ such that $p^{-1} + q^{-1} =1$. By the assumption $F\in L^\infty_TC^{\frac{1+\g}{3}}$ we have
\begin{equation}%\label{eq:}
\|F^{(\eps)} - F\|_{L^{\infty}([0,T]\times\R^d)} \to 0 \qquad \text{as }  \eps \to 0^+.
\end{equation}
Owing to this and to estimate \eqref{eq:estimate_gradient_expect}, we can let $\eps\to 0^+$ in \eqref{eq:bound_XXprime} and obtain $\Eb\big[ |X_\tau - X'_\tau |  \big]  = 0$. As $\tau\in(0,T]$ is arbitrary, this proves \eqref{eq:X_X'_equal} and concludes the proof of Theorem \ref{th:main}-(ii).
\hfill\qedsymbol{}

\subsection{Proof of Proposition %\ref{lemm:classical_sol_transp} and 
\ref{lemm:estim_transp}}\label{sec:proof_lemmas}

Let $\eps>0$ and $\tau \in(0,T]$ be fixed throughout this section. Recall that  $F^{(\eps)}(t,\cdot)$ is smooth and has sub-linear growth, uniformly with respect to $t\in[0,T]$, in light of the assumption $F \in L^\infty_TC^{\frac{1+\g}{3}}$. Therefore, 
by the method of the characteristic curves, for any $\omega\in \Omega$, the classical solution %$u_{\eps}=u_{\eps,\tau,\omega}$ 
to \eqref{eq:transport_random_eps} admits the representation
\begin{equation}\label{eq:rep_ueps_charact}
u_{\eps,\tau,\omega}(t,x)=u_{\eps}(t,x):= - \int_t^{\tau} F^{(\eps)}(s, X^{\eps,t,x}_s (\omega)) ds, \qquad (t,x)\in [0,\tau]\times\R^d,
\end{equation}
with ${X}^{\eps,\cdot,\cdot}_s (\omega)$ being the unique flow of the ordinary differential equation %the stochastic differential equation 
\begin{equation}\label{eq:peano_example_reg}
%d X^{\eps}_t = \big( W_t + F_{\eps}(X^{\eps}_t) \big) dt .
d f(s) = \big(V_s (\omega) + F^{(\eps)}(s,f(s)) \big) ds.
\end{equation}
Now fix $(t,x)\in [0,\tau]\times\R^d$ and $\omega\in\Omega$. To ease notation, below we omit the explicit dependence on $\omega$ in $u_{\eps}$ and $X^{\eps,t,x}_s$.
Following the standard theory of {ODEs (see the classic reference %\cite{ikeda2014stochastic}
\cite[Chapter 2, Section 2.4]{MR2961944})}, %\footnote{\textcolor{blue}{Da S.M a S.P e G.: mi sembra una referenza canonica. riguardo}}, 
we have
%Now, we can differentiate \eqref{eq:peano_example_reg} and obtain 
\begin{equation}\label{eq:repr_grad_flux}
\nabla_x X^{\eps,t,x}_r = I_d +  \int_t^r \nabla F^{(\eps)}(s, X^{\eps,t,x}_s) \, \nabla_x X^{\eps,t,x}_s ds, \qquad t\leq r\leq T, %\ x\in\R^d,
\end{equation}
with $\nabla F^{(\eps)}$ denoting the gradient with respect to the second variable, and Gronwall's inequality yields
\begin{equation}\label{eq:bound_grad_flux}
| \nabla_x X^{\eps,t,x}_{\tau} | \leq \exp\Big( \int_t^{\tau} |\nabla F^{(\eps)}(s,X^{\eps,t,x}_s)| ds\Big).%, \qquad 0\leq t\leq \tau, \ x\in\R^d.
\end{equation}
Differentiating \eqref{eq:rep_ueps_charact}, and then applying \eqref{eq:repr_grad_flux}, now yields
\begin{equation}\label{eq:grad_u_eps_integral}
\nabla_x u_{\eps} (t,x) = - \int_t^{\tau} \nabla F^{(\eps)}(s, X^{\eps,t,x}_s) \, \nabla_x X^{\eps,t,x}_s ds = I_d -  \nabla_x X^{\eps,t,x}_{\tau},
\end{equation}
and by \eqref{eq:bound_grad_flux} we obtain
\begin{equation}%\label{eq:}
| \nabla_x u_{\eps} (t,x) | \leq  %\int_t^T |F'( X^{\eps,t,x}_s)| \times \exp\Big( \int_t^s |F'(X^{\eps,t,x}_r)| dr\Big) ds = 
\exp\Big( \int_t^{\tau} |\nabla F^{(\eps)}(s, X^{\eps,t,x}_s)| ds\Big)+1.
\end{equation}
%for any $0\leq t\leq \tau$ and $x\in\R^d$. 
Finally, \eqref{eq:estimate_gradient_expect} stems from the lemma below. 
\hfill\qedsymbol{}

\begin{lemma}\label{lem:bound_grad_exp} For any $q\geq 0$, there exists $C_{T,q}>0$ such that
\begin{equation}\label{eq:bound_grad_exp}
\Eb\Big[\exp\Big(q \int_t^T \big| \nabla_x F^{(\eps)}(s,{X}^{\eps,t,X_t}_s)\big| ds \Big)\Big] 
\leq C_{T,q},\qquad 0\leq t \leq T, \quad \eps>0.
\end{equation}
The same estimate holds replacing $X$ with $X'$.
\end{lemma}
\begin{proof}
Fix $\eps>0$, $t\in [0,T)$, and denote by 
\begin{equation}%\label{eq:}
% (s,t,v,x)\mapsto 
({\bar V}^{t,v}_s, {\bar X}^{\eps,t,v,x}_s), \qquad  t \leq s\leq T, \quad v,x\in\R^d,
\end{equation}
the stochastic flow of the It\^o differential system \eqref{eq:sde_reg} with respect to the filtration $(\F^{W^t}_s)_{s\in[t,T]}$, the latter being the natural filtration of the shifted Brownian motion $W^t_s := W_s - W_t$, $s\in [t,T]$. This means that the process $({\bar V}^{t,v}_s, {\bar X}^{\eps,t,v,x}_s)_{s\in[t,T]}$ is the unique solution to \eqref{eq:sde_reg} with $({\bar V}^{t,v}_t, {\bar X}^{\eps,t,v,x}_t) = {(v,x)}$ and, for any $s\in (t,T]$ we have  
\begin{align}
\R^{d}\times \R^d\times \Omega \times [t,s] \ni (v,x,\omega,r)\mapsto ({\bar V}^{t,v}_r (\omega), {\bar X}^{\eps,t,v,x}_r (\omega))\\ \text{ is measurable w.r.t. }  \mathcal{B}^{d}\otimes \mathcal{B}^{d} \otimes \F^{W^t}_s\otimes \mathcal{B}_{[0,s]}.
\label{eq:measurability_property}
\end{align}
Existence and uniqueness of a flow with this property is ensured by Remark \ref{rem:strong_well_regular} together with \cite[Theorem 1.1, Chapter IV]{ikeda2014stochastic}. 
Note that ${\bar X}^{\eps,t,v,x}_s$ is not to be confused with $X^{\eps,t,x}_s$, the flow of the ordinary differential equation \eqref{eq:peano_example_reg}. 

{We now fix $q\geq0$ and $c\in(0,1)$. }
Combining estimate \eqref{eq:density_upper_bound_1dim} with \eqref{eq:moment_estimate} yields
\begin{equation}\label{eq:moment_estimate1}
{q}\int_t^{t+r} \Eb[|\nabla_x F^{(\eps)}(s,{\bar X}^{\eps,t,v,x}_s)|] ds \leq c  , \qquad  %0\leq t  \leq T, \quad
 v,x\in\R^d, \quad 0\leq r\leq \delta, 
\end{equation}
for some positive (small enough) $\delta<1$ independent of $\eps$ and $t$. 
{Notice that $\delta$ depends on $q$.} 
{Employing the Markov property {on $({\bar V}^{t,v},{\bar X}^{\eps,t,v,x})$}, 
we can follow the proof of \cite[Lemma 2.1, Chapter 1]{sznitman1998brownian} (see also \cite[Section 3]{MR2117951})
and obtain}
\begin{equation}\label{eq:exp_moment_flow_Z}
\Eb\Big[\exp\Big(q \int_t^T \big| \nabla_x F^{(\eps)}(s,{\bar X}^{\eps,t,v,x}_s)\big| ds \Big)\Big]  \leq C_{T,q}, \qquad v,x\in\R^{d},
\end{equation}
with $C_{T,q}>0$ independent of $\eps$ and $t$.
Relying on \eqref{eq:measurability_property} and since the filtration $(\F^{W^t}_s)_{s\in[t,T]}$ is independent of $(V_t,X_t)$, the Freezing Lemma {(see \cite[Theorem 5.2.7-13)]{Pascucci2024book})} 
also yields %$(V_t,X_t)$ is independent of $(\F_s)_{s\in[]}$, the $\s$-algebra of the Brownian increments after time $t$, 
%....,
%we also have
%\eqref{eq:exp_moment_flow_Z} yields 
\begin{align}%\label{eq:}
&\Eb\Big[\exp\Big(q \int_t^T \big| \nabla_x F^{(\eps)}(s,{\bar X_s}^{\eps,t,V_t,X_t})\big| ds \Big)\Big]\\
&= \Eb\Big[ \Eb\Big[\exp\Big(q \int_t^T \big| \nabla_x F^{(\eps)}(s,{\bar X_s}^{\eps,t,v,x})\big| ds \Big)\Big]\Big|_{(v,x)=(V_t,X_t)}\Big].   
%\leq C_{T,q}.
\end{align}
%for $T$ sufficiently small. 
This, together with \eqref{eq:exp_moment_flow_Z}, proves \eqref{eq:bound_grad_exp} upon observing that $X^{\eps,t,X_t} = {\bar X}^{\eps,t,V_t,X_t}$ almost surely for they both solve the same (regular) ODE.
\end{proof}

\bibliographystyle{acm}
\bibliography{Bibtex-Final}

\end{document}